\documentclass[a4paper,12pt]{article}
\usepackage{amsmath}
\usepackage{amsfonts}
\usepackage{amssymb}
\usepackage{amsthm}
\usepackage{hyperref}
\newtheorem{theorem}{Theorem}[section]
\newtheorem{lemma}[theorem]{Lemma}

\newtheorem{corollary}[theorem]{Corollary}

\newtheorem{definition}[theorem]{Definition}

\newtheorem{question}[theorem]{Question}

\allowdisplaybreaks
\def\e{\epsilon}

\def\g{\gamma}
\def\d{\delta}

\def\R{\mathbb{R}}

\def\C{\mathbb{C}}
\def\E{\mathbb{E}}

\def\tr{\mbox{tr}}
\def\ve{\varepsilon}
\def\op{\mathrm{op}}
\def\nuc{\mathrm{nuc}}

\usepackage[utf8]{inputenc}
\title{Inverse and stability theorems for approximate representations of finite groups}
\author{W.T. Gowers\thanks{Department of Pure Mathematics and Mathematical Statistics, Wilberforce Road, Cambridge CB3 0WB, UK. E-mail: {\tt w.t.gowers@dpmms.cam.ac.uk}. Research supported by a Royal Society 2010 Anniversary Research Professorship.} \and
O. Hatami\thanks{Department of Pure Mathematics and Mathematical Statistics, Wilberforce Road, Cambridge CB3 0WB, UK. E-mail: {\tt oh233@cam.ac.uk}. Research supported by a Cambridge Trust Scholarship and by the Department of Pure Mathematics and Mathematical Statistics, Cambridge.}}
\date{}
\begin{document}

\maketitle

\begin{abstract}
The $U^2$ norm gives a useful measure of quasirandomness for real- or complex-valued functions defined on finite (or, more generally, locally compact) groups. A simple Fourier-analytic argument yields an inverse theorem, which shows that a bounded function with a large $U^2$ norm defined on a finite Abelian group must correlate significantly with a character. In this paper we generalize this statement to functions that are defined on arbitrary finite groups and that take values in M$_n(\C)$. The conclusion now is that the function correlates with a representation -- though with the twist that the dimension of the representation is shown to be within a constant of $n$ rather than being exactly equal to $n$. There are easy examples that show that this weakening of the obvious conclusion is necessary. The proof is much less straightforward than it is in the case of scalar functions on Abelian groups.

As an easy corollary, we prove a stability theorem for near representations. It states that if $G$ is a finite group and $f:G\to$M$_n(\C)$ is a function that is close to a representation in the sense that $f(xy)-f(x)f(y)$ has a small Hilbert-Schmidt norm (also known as the Frobenius norm) for every $x,y\in G$, then there must be a representation $\rho$ such that $f(x)-\rho(x)$ has small Hilbert-Schmidt norm for every $x$. Again, the dimension of $\rho$ need not be exactly $n$, but it must be close to $n$. We also obtain stability theorems for other Schatten $p$-norms. A stability theorem of this kind was obtained for the operator norm by Grove, Karcher and Ruh in 1974 \cite{Grove} and in a more general form by Kazhdan in 1982 \cite{Kazhdan}. (For the operator norm, the dimension of the approximating representation is exactly $n$.)
\end{abstract}
\section{Introduction}

Let $G$ be a finite Abelian group. The $U^2$-\emph{norm} of a function $f:G\to\C$ is defined by the formula
\[\|f\|_{U^2}^4=\mathbb E_{x-y=z-w}f(x){\overline{f(y)}}\,\overline{f(z)}f(w),\]
where $\E_{x-y=z-w}$ denotes the average over all quadruples $(x,y,z,w)$ such that $x-y=z-w$. Note that we can rewrite the right-hand side as $\E_{x+y=z+w}f(x)f(y)\overline{f(z)f(w)}$. The $U^2$ norm is a useful measure of quasirandomness: for example, if $A$ is a random subset of $G$, then with high probability its characteristic function is close in the $U^2$ norm to the constant function that takes the value 1/2 everywhere.

The \emph{discrete Fourier transform} of $f$ is the function $\hat f:\hat G\to\C$ defined by the formula
\[\hat f(\chi)=\E_xf(x)\overline{\chi(x)}.\]
Here, $\hat G$ is the dual group of $G$ and $\chi$, a typical element of $\hat G$, is a character. 

As is customary, we define concepts such as convolutions and $p$-norms for functions defined on $G$ using the uniform probability measure, whereas for functions defined on $\hat G$ we use counting measure. Thus, if $f,g:G\to\C$, then $\|f\|_p=(\E_x|f(x)|^p)^{1/p}$ and $f*g(x)=\E_{y+z=x}f(y)g(z)$, while $\|\hat f\|_p=(\sum_\chi|\hat f(\chi)|^p)^{1/p}$ and $\hat f*\hat g(\chi)=\sum_{\chi_1\chi_2=\chi}\hat f(\chi_1)\hat g(\chi_2)$.

It is a straightforward exercise to prove the basic properties of the discrete Fourier transform, of which the main ones are the following. 
\begin{itemize}
\item \emph{Parseval's identity} states that $\langle\hat f,\hat g\rangle=\langle f\,g\rangle$, from which it follows that $\|\hat f\|_2=\|f\|_2$.
\item The \emph{convolution law} states that $\widehat{f*g}(\chi)=\hat f(\chi)\hat g(\chi)$.
\item The \emph{inversion theorem} states that $f(x)=\sum_\chi\hat f(\chi)\chi(x)$.
\end{itemize}
This gives us that
\[\|f\|_{U^2}^4=\E_{x+y=z+w}f(x)f(y)\overline{f(z)f(w)}=\langle f*f,f*f\rangle=\sum_\chi|\hat f(\chi)|^4=\|\hat f\|_4^4.\]
Suppose we know that $\|f\|_\infty\leq 1$ and $\|f\|_{U^2}\geq c$. Then $|\hat f(\chi)|\leq 1$ for every $\chi$, $\|\hat f\|_2^2=\|f\|_2^2\leq 1$, and $\sum_\chi|\hat f(\chi)|^4\geq c$. It follows that 
\[c\leq\sum_\chi|\hat f(\chi)|^4\leq\max_\chi|\hat f(\chi)|^2\|\hat f\|_2^2\leq\max_\chi|\hat f(\chi)|^2.\]
Therefore, there exists $\chi$ such that 
\[|\hat f(\chi)|=|\E_xf(x)\overline{\chi(x)}|\geq c^{1/2}.\]
This is an easy and well-known example of an \emph{inverse theorem}. It tells us that if a bounded function has a large $U^2$ norm, then it must correlate with a character. The converse is even easier: if $|\hat f(\chi)|\geq c$, then 
\[\|f\|_{U^2}=\|\hat f\|_4\geq\|\hat f\|_\infty\geq c.\]

In this paper, we shall generalize this result in two directions simultaneously: we shall allow $G$ to be a non-Abelian group, and we shall allow the function $f$ to take values in a matrix group M$_n(\C)$. The appropriate definition of the $U^2$ norm turns out to be given by the formula
\[\|f\|_{U^2}^4=\E_{xy^{-1}zw^{-1}}\tr(f(x)f(y)^*f(z)f(w)^*),\]
the appropriate definition of the $L_\infty$ norm is 
\[\|f\|_\infty=\max_x\|f(x)\|_\op,\]
where $\|.\|_\op$ is the operator norm, and the role played by characters is now played by (translates of) unitary representations -- that is, (Freiman) homomorphisms from $G$ to a unitary group. 

Our main theorem (Theorem \ref{inversetheorem} below) will tell us that a function $f:G\to$M$_n(\C)$ with $\|f\|_\infty\leq 1$ and $\|f\|_{U^2}^4\geq cn$ must correlate, in a suitable sense, with a unitary representation of dimension between $c_1n$ and $c_2n$, where $c_1$ and $c_2$ are constants that depend on $c$ only. Note that if one were to ask for a representation of dimension exactly $n$, then the theorem would become obviously false, since it may be that the only $n$-dimensional representation of $G$ is a sum of $n$ copies of the trivial representation.

In order to prove this theorem, we use a natural generalization of Fourier analysis, which first appeared (to the best of our knowledge) in a paper of Moore and Russell \cite{Moore}. However, the proof is considerably less straightforward than the proof for Abelian groups and scalar-valued functions, because although that argument can be generalized in a natural way, the resulting generalization yields a conclusion that is much weaker than we need. Roughly speaking, it tells us that many small pieces of our function $f$ correlate with irreducible representations. There then remains the task of finding a way to move these small irreducible representations so that they become orthogonal and can be put together into a large representation that still correlates with~$f$.

As a by-product of our inverse theorem, we obtain as a straightforward consequence a \emph{stability theorem} for unitary representations. This result fits into a program initiated by Ulam in 1940, who asked the following general question.

\begin{question} Let $G_1$ be a group and let $G_2$ be a metric group with a metric~$d$. Given $\varepsilon>0$, does there exist $\delta>0$ such that if a function $f:G_1\rightarrow G_2$ satisfies the inequality $d(f(xy),f(x)f(y))< \delta$ for all $x,y\in G_1$, then there is a homomorphism $g:G_1\rightarrow G_2$ such that $d(f(x),g(x))<\varepsilon$ for all $x\in G_1$.
\end{question}

\noindent If the answer is yes, then one says that the functional equation that defines the homomorphism property is \emph{stable}. 

We shall be interested in the case where $G_1$ is a finite group and $G_2$ is a unitary group $U(n)$. The metric we shall take on $U(n)$ is given by the \emph{Hilbert Schmidt} norm (also known as the \emph{Frobenius} norm). This norm is defined on M$_n(\C)$ by the formula $\|A\|_{HS}^2=\tr(AA^*)$. If we think in matrix terms, then it is also given by the formula $\sum_{i,j}|A_{ij}|^2$. It comes with the inner product $\langle A,B\rangle=\tr(AB^*)$.

Let $G$ be a finite group. A \emph{unitary representation} of $G$ is a homomorphism $\rho:G\to U(H)$ for some Hilbert space $H$, where $U(H)$ is the group of unitary operators on $H$. An \emph{approximate unitary representation} is a map $f:G\to U(H)$ such that $f(gh)$ is approximately equal to $f(g)f(h)$ for any two elements $g,h\in G$, where the approximation is in some suitable norm. For any matrix norm $\|.\|$ that is invariant under taking adjoints and under multiplication by a unitary map, there is a simple class of examples: take a unitary representation $\rho$ and take any function $f:G\to U(H)$ such that $\|f(g)-\rho(g)\|$ is small for every $g$. It is natural to ask whether \emph{all} examples are of this form, which is precisely Ulam's stability problem when $G_1=G$, $G_2=U(H)$, and the metric on $G_2$ is given by the matrix norm. 

An additional point is that when $H$ is finite dimensional, it is desirable to normalize our matrix norms so that the dependence of $\d$ on $\ve$ is independent of the dimension of $H$. For example, with the Hilbert-Schmidt norm we can do this by defining $\|A\|_{hs}^2$ to be $n^{-1}\tr(AA^*)$, or equivalently $\|A\|_{hs}=n^{-1/2}\|A\|_{HS}$, where $n=\dim H$. The diameter of $U(H)$ is 2 for this normalized version of the norm, instead of $2n^{1/2}$ for the unnormalized version. For our main results, we wish to regard two $n\times n$ matrices $A$ and $B$ as close if $\|A-B\|_{HS}\leq\ve\sqrt n$ for some small $\ve$, and it is more natural to express this condition by writing it as $\|A-B\|_{hs}\leq\ve$.
 
Results of this kind have been known for some time when the matrix norm in question is the operator norm $\|A\|_\op=\max\{\|Ax\|:\|x\|=1\}$. In 1974, Grove, Karcher and Ruh proved \cite{Grove} that unitary representations of compact groups are stable with respect to the operator norm. This result was rediscovered by Kazhdan in 1982 and generalized to amenable groups. Kazhdan's version is as follows.

\begin{theorem} [Kazhdan]\label{kazhdan}
Let G be an amenable group and let $f: G\rightarrow U(H)$ for some Hilbert space $H$. Let $\varepsilon<\frac 1{200}$ and suppose that $\|f(gh)-f(g)f(h)\|_{\op}\leq \varepsilon$ for all $g,h\in G$. Then there exists a representation $\rho : G \rightarrow U(H)$ such that $\|f(g)-\rho(g)\|_{\op}<2\varepsilon$ for every $g\in G$.
\end{theorem}

A short proof can be found in \cite{Shtern} or \cite{Burger}. In a slightly earlier paper, also from 1974, Grove, Karcher and Ruh proved \cite{Grove1} a theorem that implies a stability result for the Hilbert-Schmidt norm. (It appears as Theorem 4.3 in their paper.) Rephrasing the stability result in terms of the normalized Hilbert-Schmidt norm, we can state it as follows.

\begin{theorem}[Grove, Karcher, Ruh] \label{GKR} Let $G$ be a compact Lie group, let $\ve\leq(\pi/6)n^{-1/2}$, and let $f:G\to U(n)$ be a map such that $\|f(xy)-f(x)f(y)\|_{hs}\leq\ve$ for every $x,y\in G$. Then there exists a representation $\rho:G\to U(n)$ such that $\|f(x)-\rho(x)\|_{hs}\leq 1.36\ve$ for every $x\in G$. 
\end{theorem}

A similar result also appears in an unpublished preprint from 2003, by Babai, Friedl and Luk\'acs \cite{Babai}. It is somewhat weaker than the result of Grove, Karcher and Ruh just mentioned, since they require a smaller upper bound on $\ve$ and the constant of proportionality they obtain depends on $n$. However, they introduced some interesting techniques that have influenced the methods we use in this paper. 

For some applications it would be highly desirable to be able to prove a dimension-independent result. That is, we would like the result to apply to all sufficiently small $\varepsilon$, where the smallness condition is independent of $n$. Our interest in the problem arose because we needed precisely such a statement in order to prove another theorem.

If one wishes to improve Theorem \ref{GKR} in this way, then one has to face up to an example that is initially rather discouraging. Let $G$ be a finite group, let $n+1$ be the smallest dimension of an irreducible representation, and suppose that $n$ is large. Let $\rho:G\to U(n+1)$ be an irreducible representation and let $\pi:\C^{n+1}\to\C^n$ be an orthogonal projection with adjoint $\iota$ (which is an insertion map from $\C^n$ to $\C^{n+1}$). Finally, let $f:G\to U(n)$ be defined by the formula $f(x)=\pi\rho(x)\iota$ for each $x$.

Then $f(x)f(y)=\pi\rho(x)\iota\pi\rho(y)\iota$. But with respect to a suitable orthonormal basis, $\iota\pi$ is an $(n+1)\times(n+1)$ diagonal matrix with 1s everywhere except in the final position where there is a zero. It follows straightforwardly that 
\[\pi\rho(x)\iota\pi\rho(y)\iota\approx\pi\rho(x)\rho(y)\iota=\pi\rho(xy)\iota=f(xy),\]
where the approximation is in the normalized Hilbert-Schmidt norm. Moreover, the error tends to zero with $n$.

It is clear that in some sense the representation that approximates $f$ ought to be $\rho$, but $\rho$ is of the wrong dimension. Moreover, there are \emph{no} non-trivial representations of dimension $n$, and the trivial representation is a very bad approximation indeed. So when the smallest nontrivial representation of $G$ has a large dimension (such groups are called \emph{quasirandom} in \cite{Gowers}), there are approximate representations with very small $\ve$ that cannot be approximated even crudely by a representation.

It follows that the requirement in Theorem \ref{GKR} that $\varepsilon$ should be bounded above by a function of $n$ is necessary, and with a bit more care one can show that that function cannot be substantially better than the $n^{-1/2}$ that Grove, Karcher and Ruh obtained. However, if we do not insist that the approximating representation is of the same dimension as $f$, then this conclusion no longer follows. Given that the normalized Hilbert-Schmidt norm is insensitive to low-rank perturbations, it is not very natural to insist that the approximating representation should have the same dimension as $f$. Our main theorem takes this into account, and can therefore be seen as the ``correct" version of the Ulam stability problem for the Hilbert-Schmidt norm when $\ve$ is significantly larger than $n^{-1/2}$. It can be stated imprecisely as follows. (The precise statement appears as Theorem \ref{reppcase} but depends on some definitions that we give later.) If $f:G\to U(n)$ is a map such that $\|f(x)f(y)-f(xy)\|_{hs}\leq\ve$ for every $x,y\in G$, then 
there exists $m$ close to $n$ and a representation $\rho:G\to U(m)$ such that $\|f(x)-\rho(x)\|_{hs}\leq C\ve$ for every $x$, where $C$ is an absolute constant, and we interpret $\|f(x)-\rho(x)\|_{hs}$ in a natural way (to allow for the fact that they are matrices of slightly different dimension). The constant $C$ we obtain is considerably worse than the $1.36$ obtained by Grove, Karcher and Ruh, but our result is applicable for all $\ve$. The error in the dimension in our result is proportional to $\ve^2n$, so when $\ve\leq cn^{-1/2}$ the representation $\rho$ has the same dimension as the approximate representation $f$. Thus, our result implies the result of Grove, Karcher and Ruh, apart from the worsening of the constants.

The paper is organized as follows. In the next section we collect together some basic lemmas about matrices that will be used without comment in the rest of the paper. In Section 3 we recall the definition and basic properties of the Fourier transform we need for matrix-valued functions on general finite groups, and in Section 4 we apply it to obtain many approximately invariant subspaces. In Section 5 we show how to piece together suitable modifications of these subspaces to create a representation that correlates with the original function and thereby gives us our inverse theorem. In Section 6 we show how to deduce the stability results from the inverse theorem. In Section 7 we prove that the approximating representation we find is, in a suitable sense, unique. (It is not precisely unique, since two distinct representations can be close in the normalized Hilbert-Schmidt norm.) Finally, in Section 8 we make a few concluding remarks and mention some questions to which we do not know the answers.

After we posted a first version of this paper to the arXiv, Narutaka Ozawa informed us in a private communication that the stability result could also be proved by operator-algebraic methods. The proof he outlined, which is short, given various known results in functional analysis, gives a better dependence on $\ve$ for the $p$-Schatten norms when $2<p<\infty$. His argument does not yield the inverse theorem: it would be interesting to know whether an operator-algebraic approach could be made to work for that too.

\section{A few preliminaries}

In this section we shall briefly introduce some of the principal definitions (which are mostly standard) that will be used throughout the paper.

\subsection{Singular values and matrix norms}

We have already mentioned the Hilbert-Schmidt norm. This plays the the role for matrices that the $\ell_2$ norm plays for real or complex-valued functions defined on finite groups. Another important norm on such functions is the $U^2$ norm, which we have also mentioned. For matrices its role is played by the \emph{box norm}, which can be defined by the formula $\|A\|_\square^4=\tr(AA^*AA^*)=\|AA^*\|_{HS}^2$. (It is not too hard to prove that this formula defines a norm -- it will in fact follow from a lemma proved later in this section.) 

Fourier coefficients play a very important role in basic additive combinatorics. For matrices a similar role is played by singular values. Recall that if $V$ and $W$ are complex inner product spaces of dimensions $n$ and $m$, respectively, $A:V\to W$ is a linear map, and $r=\min\{m,n\}$, then there exist orthonormal bases $v_1,\dots,v_n$ of $V$ and $w_1,\dots,w_m$ of $W$ and non-negative real numbers $\lambda_1,\dots,\lambda_r$ such that $Av_i=\lambda_iw_i$ for $i=1,\dots,r$ and $Av_i=0$ if $i>r$. The numbers $\lambda_1,\dots,\lambda_r$ are called the \emph{singular values} of $A$. An equivalent statement in terms of matrices is that if $A$ is an $m\times n$ complex matrix, then there exist unitary matrices $P\in U(m)$ and $Q\in U(n)$ such that $PAQ$ is diagonal, in the sense that $(PAQ)_{ij}=0$ whenever $i\ne j$. The diagonal entries $(PAQ)_{ii}$ are unique up to permutation and are equal to the singular values.

From this uniqueness it follows that the singular values are unaffected if we multiply $A$ on the left or right by a unitary matrix. The same is easily seen to be true of $\tr(AA^*)$ and $\tr(AA^*AA^*)$. If $A$ is diagonal with non-negative entries, then $\tr(AA^*)$ is the sum of the squares of those entries, and $\tr(AA^*AA^*)$ is the sum of their fourth powers. It follows that in general $\|A\|_{HS}^2$ is the sum of the squares of the singular values of $A$ and $\|A\|_\square^4$ is the sum of the fourth powers of the singular values. Also, $\|A\|_{\op}$ is the maximum singular value. It is straightforward to check that if $G$ is a group and $f:G\to\C$, then the singular values of the corresponding convolution operator $A_f$ are the absolute values of the Fourier coefficients of~$f$, which explains why singular values of matrices have several properties that are similar to properties of Fourier coefficients of scalar-valued functions on finite groups.

Another norm we shall consider later is the \emph{nuclear norm}. The nuclear norm $\|A\|_\nuc$ of a matrix $A$ is defined to be the sum of its singular values. Equivalently, it is the smallest possible value of $\sum_i\lambda_i$ such that every $\lambda_i\geq 0$ and we can write $A=\sum_i\lambda_ia_i\otimes b_i$ for unit vectors $a_i$ and $b_i$. That is, the unit ball of the nuclear norm is the convex hull of the rank-1 matrices of norm 1. The nuclear norm is the matrix equivalent of the $\ell_1$ norm.

We begin with a matrix version of the $\ell_1$-$\ell_\infty$ inequality.

\begin{lemma} \label{nuclearop}
Let $A$ and $B$ be $n\times m$ matrices. Then $|\tr(AB^*)|\leq\|A\|_\op\|B\|_\nuc$.
\end{lemma}

\begin{proof}
If $A$ is a matrix and $a\otimes b$ is a rank-1 matrix, then $A(a\otimes b)_{ij}=\sum_kA_{ik}a_kb_j=(Aa)_ib_j$. That is, $A(a\otimes b)=Aa\otimes b$. It follows that 
\[\tr(A(a\otimes\overline{b}))=\tr(Aa\otimes\overline{b})=\langle Aa,b\rangle\leq\|Aa\|_2\|b\|_2\leq\|A\|_\op.\]
The result now follows from the triangle inequality and the second definition above of the nuclear norm.
\end{proof}

It is not hard to see that in fact the nuclear and operator norms are dual to each other. Indeed, if $A$ has operator norm 1, then pick unit vectors $u$ and $v$ such that $Au=v$. Then $v\otimes \overline u$ has nuclear norm 1 and 
\[\tr(A(v\otimes \overline u)^*)=\sum_{i,j}A_{ij}\overline{v_i}u_j=\langle Au,v\rangle=1.\]

The next lemma will be useful for providing a sort of bridge between non-square matrices and square matrices.

\begin{lemma} \label{partialunitary}
Let $U$ be an $n\times m$ matrix with all its singular values equal to 1. Then if $n\leq m$ the rows of $U$ form an orthonormal sequence, and if $n\geq m$ the columns form an orthonormal sequence.
\end{lemma}

\begin{proof}
Suppose first that $n\leq m$. Then there are $n$ singular values, so we can write $U$ in the form $\sum_{r=1}^na(r)\otimes\overline{b(r)}$, where $a(1),\dots,a(n)$ and $b(1),\dots,b(n)$ are orthonormal sequences in $\C^n$ and $\C^m$, respectively. Then the inner product of the $i$th row of $U$ with the $j$th row of $U$ is
\[\sum_kU_{ik}\overline{U_{jk}}=\sum_k\sum_{r,s}a(r)_i\overline{b(r)_k}\overline{a(s)_j}b(s)_k.\]
Since the $b(r)$ are orthonormal, this is equal to $\sum_ra(r)_ia(r)_j$. But $a(1),\dots,a(n)$ form the rows of an $n\times n$ unitary matrix, which therefore has orthonormal columns, so this last sum is 1 if $i=j$ and 0 otherwise. 

This proves the result when $n\leq m$. If $n>m$ then we can apply the above argument to $U^*$.
\end{proof}

Let us define a \emph{partial unitary matrix} to be an $n\times m$ matrix with orthonormal rows if $n\leq m$ and orthonormal columns if $n\geq m$. The reason for this terminology is that a partial unitary matrix can be extended to a unitary matrix of dimension $\max\{n,m\}$ by the addition of some more rows or columns.

\subsection{A $U^2$-norm for matrix-valued functions}

Let $G$ be a finite Abelian group and let $f:G\to\C$. The $U^2$-\emph{norm} of $f$ is defined by the formula
\[\|f\|_{U^2}^4=\E_{x+y=z+w}f(x)f(y)\overline{f(z)f(w)}.\]
It is not too hard to show that this does indeed define a norm.

A convenient generalization of the definition for non-Abelian finite groups turns out to be the following.
\[\|f\|_{U^2}^4=\E_{xy^{-1}zw^{-1}=e}f(x)\overline{f(y)}f(z)\overline{f(w)}.\]

A convenient further generalization of this for functions $f:G\to $M$_n(\C)$~is
\[\|f\|_{U^2}^4=\E_{xy^{-1}zw^{-1}=e}\tr(f(x)f(y)^*f(z)f(w)^*).\]
It will sometimes also be useful to consider a normalized version of this definition, namely
\[\|f\|_{u^2}^4=n^{-1}\E_{xy^{-1}zw^{-1}=e}\tr(f(x)f(y)^*f(z)f(w)^*).\]
As with the scalar $U^2$-norms, it is also useful to define a kind of generalized inner product. We set
\[[f_1,f_2,f_3,f_4]=\E_{xy^{-1}zw^{-1}=e}\mathrm{tr}(f_1(x)f_2(y)^*f_3(z)f_4(w)^*).\]
Again as in the scalar case, this generalized inner product comes with a generalized Cauchy-Schwarz inequality.

\begin{lemma} \label{matrixU2}
Let $f_1,f_2,f_3,f_4$ be functions from $G$ to $U(n)$. Then
\[|[f_1,f_2,f_3,f_4]|\leq\|f_1\|_{U^2}\|f_2\|_{U^2}\|f_3\|_{U^2}\|f_4\|_{U^2}.\]
\end{lemma}

\begin{proof}
This is proved in a standard way using repeated applications of the usual Cauchy-Schwarz inequality. We have
\begin{eqnarray*}
[f_1,f_2,f_3,f_4]&=&\E_u\tr\bigl((\E_{xy^{-1}=u}f_1(x)f_2(y)^*)(\E_{wz^{-1}=u}f_4(w)f_3(z)^*)^*\bigr)\\
&\leq&\E_u\|\E_{xy^{-1}=u}f_1(x)f_2(y)^*\|_{HS}\|\E_{wz^{-1}=u}f_4(w)f_3(z)^*\|_{HS}\\
&\leq&(\E_u\|\E_{xy^{-1}=u}f_1(x)f_2(y)^*\|_{HS}^2)^{1/2}(\E_u\|\E_{wz^{-1}=u}f_4(w)f_3(z)^*\|_{HS}^2)^{1/2}\\
&=&[f_1,f_2,f_2,f_1]^{1/2}[f_4,f_3,f_3,f_4]^{1/2}.\\
\end{eqnarray*}
By the cyclic property of the trace, $[f_1,f_2,f_3,f_4]$ is the complex conjugate of $[f_4,f_1,f_2,f_3]$. Also, $[f,g,g,f]$ is always real and non-negative, by the first equality above and the fact that for any matrix $\tr(AA^*)=\|A\|_{HS}^2\geq 0$. Therefore,
\[[f_1,f_2,f_2,f_1]=[f_1,f_1,f_2,f_2].\]
By the inequality we have just proved, it follows that
\[[f_1,f_2,f_2,f_1]\leq[f_1,f_1,f_1,f_1]^{1/2}[f_2,f_2,f_2,f_2]^{1/2}=\|f_1\|_{U^2}^2\|f_2\|_{U^2}^2\]
with a similar inequality for $[f_3,f_4,f_4,f_3]$. Putting all this together gives the result.
\end{proof}

Writing $[f_1,f_2,f_3,f_4]'$ for $n^{-1}[f_1,f_2,f_3,f_4]$, we also have the inequality
\[|[f_1,f_2,f_3,f_4]'|\leq\|f_1\|_{u^2}\|f_2\|_{u^2}\|f_3\|_{u^2}\|f_4\|_{u^2}.\]

\begin{corollary}
The functions $\|.\|_{U^2}$ and $\|.\|_{u^2}$ really are norms.
\end{corollary}

\begin{proof}
The only non-trivial part of this is the triangle inequality. But 
\[\|f+g\|_{U^2}^4=[f+g,f+g,f+g,f+g]\]
is a sum of 16 terms of the form $[f_1,f_2,f_3,f_4]$ where each $f_i$ is $f$ or $g$. Using Lemma \ref{matrixU2} we can therefore bound the right-hand side by $(\|f\|_{U^2}+\|g\|_{U^2})^4$, and we are done. The result for $\|.\|_{u^2}$ is an immediate consequence.
\end{proof}

The next result is a (much easier) converse to the inverse theorem that we shall prove later. Our aim later will be to prove that every function $f:G\to$M$_n(\C)$ that takes values with operator norm at most 1 and satisfies $\|f\|_{u^2}\geq c$ must correlate with a representation of dimension not too different from $n$. Here we show that this condition is sufficient as well as necessary. If $A$ and $B$ are matrices of the same size, we shall write $\langle A,B\rangle$ for the matrix inner product $\tr(AB^*)=\sum_{ij}A_{ij}\overline{B_{ij}}$.

\begin{corollary} \label{inverseconverse}
Let $m$ and $n$ be positive integers, let $f:G\to$M$_n(\C)$ be a function, let $c>0$, let $U$ and $V$ be $n\times m$ partial unitary matrices, let $P:G\to U(m)$ be a representation, and suppose that
\[|\E_x\langle f(x),VP(x)U^*\rangle|\geq cm.\]
Then $\|f\|_{U^2}^4\geq c^4m$.
\end{corollary}

\begin{proof}
Suppose first that $n\geq m$. Then $U^*U=V^*V=I_m$, from which it follows that if $wz^{-1}y=x$, then 
\[VP(x)U^*=(VP(w)U^*)(VP(z)U^*)^*(VP(y)U^*).\]
Therefore, setting $\sigma(x)=VP(x)U^*$ for each $x$, we have that
\[\E_x\langle f(x),VP(x)U^*\rangle=\E_{xy^{-1}zw^{-1}=e}\langle f(x),\sigma(w)\sigma(z)^*\sigma(y)\rangle=[f,\sigma,\sigma,\sigma].\]
Also, if $xy^{-1}zw^{-1}=e$, then 
\[\sigma(x)\sigma(y)^*\sigma(z)\sigma(w)^*=V\sigma(x)\sigma(y)^*\sigma(z)\sigma(w)^*V^*=VV^*.\]
But $\tr(VV^*)=\tr(V^*V)=m$, so $\|\sigma\|_{U^2}^4=m$. Therefore, by Corollary \ref{matrixU2} we get that $\|f\|_{U^2}\geq cm/m^{3/4}$, which proves the result.

If $n\leq m$, then we can rewrite the initial inequality as
\[|\E_x\langle V^*f(x)U,P(x)\rangle|\geq cm.\]
Let $g(x)=V^*f(x)U$ for each $x$. This time we have $P(x)=P(w)P(z)^*P(y)$ whenever $yz^{-1}w=x$, so this inequality is equivalent to the statement that $|[g,P,P,P]|\geq cm$. Since $\|P\|_{U^2}=m^{1/4}$, it follows from Corollary \ref{matrixU2} that $\|g\|_{U^2}\geq cm^{1/4}$. Also, since $UU^*=VV^*=I_n$, for any $x,y,z,w$ we have
\begin{align*}
\tr(g(x)g(y)^*g(z)g(w)^*)&=\tr(V^*f(x)f(y)^*f(z)f(w)^*V)\\
&=\tr(f(x)f(y)^*f(z)f(w)^*),\\
\end{align*}
from which it follows that $\|g\|_{U^2}=\|f\|_{U^2}$. So we have the result in this case as well.
\end{proof}

The next result is a slightly different way of expressing the same basic idea: that a function that correlates with a representation of a not too different dimension has a $u_2$ norm that is bounded below.

\begin{corollary}
Let $f:G\to$M$_n(\C)$ be a function, let $c>0$, and let $\rho:G\to U(m)$ be a representation such that $\E_x\tr'(f(x)(\rho(x)\oplus 0_{n-m})^*)\geq c$ if $m\leq n$ and $\E_x\tr'((f(x)\oplus 0_{m-n})\rho(x)^*)\geq c$ if $m>n$. Then $\|f\|_{u_2}\geq c(n/m)^{3/4}$ in the first case, and $\|f\|_{u^2}\geq c(m/n)^{1/4}$ in the second.
\end{corollary}

\begin{proof}
Let $g=f$ and $\sigma=\rho\oplus 0_{n-m}$ in the first case and let $g=f\oplus 0_{m-n}$ and $\sigma=\rho$ in the second case. In both cases, $\sigma(x)=\E_{wz^{-1}y=x}\sigma(w)\sigma(z)^*\sigma(y)$. It follows that
\begin{eqnarray*}c&\leq&\E_x\tr'(g(x)\sigma(x)^*)\\
&=&\E_{xy^{-1}zw^{-1}=e}g(x)\sigma(y)^*\sigma(z)\sigma(w)^*\\
&=&[g,\sigma,\sigma,\sigma]\\
&\leq&\|g\|_{u^2}\|\sigma\|_{u^2}^3.\\
\end{eqnarray*}
If $m\leq n$, then $f=g$ and $\|\sigma\|_{U^2}^4=\|\rho\|_{U^2}^4=m$, so $\|\sigma\|_{u^2}^4=m/n$, and the result follows. If $m>n$, then $\|\sigma\|_{u^2}=1$, so $\|g\|_{u^2}^4\geq c^4$, which implies that $\|f\|_{u^2}^4\geq c^4m/n$ and again the result follows.
\end{proof}

\subsection{An inequality concerning real sequences}

Later on, we shall obtain a collection of singular values from which we shall want to pick out the large ones, in an appropriate sense. The following lemma will help us with this.

\begin{lemma}\label{singularvalues} Let $a_1,\dots,a_m$ be real numbers belonging to the interval $[0,1]$, let $n$ be a positive integer, and let $0\leq c\leq 1$. Suppose that $\sum_{i=1}^ma_i=n$ and $\sum_{i=1}^ma_i^2\geq cn$. Let $A=\{i:a_i\geq c/2\}$. Then $cn/(2-c)\leq|A|\leq(2-c)n/c$ and $\sum_{i\in A}a_i\geq cn/(2-c)$.
\end{lemma}

\begin{proof}
We have
\[\sum_{i\notin A}a_i^2\leq(c/2)\sum_{i\notin A}a_i=(c/2)(n-\sum_{i\in A}a_i).\]
Therefore,
\[cn\leq\sum_ia_i^2\leq\sum_{i\in A}a_i+\frac c2(n-\sum_{i\in A}a_i)=\frac{cn}2+(1-c/2)\sum_{i\in A}a_i.\]
It follows that
\[\sum_{i\in A}a_i\geq \frac{cn}{2-c}.\]
From this it follows that $|A|\geq cn/(2-c)$ as well.

For the upper bound on $|A|$, let us assume that $|A|=m$ and try to maximize $\sum_{i\in A}a_i^2$ subject to the constraints that $a_i\geq c/2$ for every $i\in A$ and $\sum_{i\in A}a_i\leq n$. If $a\geq b$ and $\d>0$, then $(a+\d)^2+(b-\d)^2>a^2+b^2$. From this it follows that $\sum_{i\in A}a_i^2$ is at most $r+(c/2)^2(m-r)$, where $r$ is chosen such that $r+(c/2)(m-r)=n$. (It does not matter if this $r$ is not an integer -- we can generalize in an obvious way to step functions defined on $\R$.) 

But if $r+(c/2)(m-r)=n$, then $r=(n-cm/2)/(1-c/2)$, so 
\begin{align*}
r+(c/2)^2(m-r)&=c^2m/4+r(1-c^2/4)\\
&=c^2m/4+(n-cm/2)(1+c/2)\\
&=n+c(n-m)/2\\
&=n(1+c/2)-cm/2.\\
\end{align*}
For this to be at least $cn$, we need $cm/2\leq n(1-c/2)$, or $m\leq n(2-c)/c$.
\end{proof}

We remark that for $c$ small we shall use the bounds $cn/2\leq|A|\leq 2n/c$ and $\sum_{i=1}^na_i\geq cn/2$, which are easier to prove. The reason we worked a little bit more to obtain the stronger bounds is that we shall use the lemma when $c$ is close to 1. If $c=1-\e$ for small $\e$, then $c/(2-c)=(1-\e)/(1+\e)\geq 1-2\e$. It will matter to us that the difference between this and 1 is $O(\e)$. 

\begin{corollary}\label{weightedsingularvalues} Let $m$ and $n$ be positive integers, let $a_1,\dots,a_m\in [0,1]$ be real numbers and let $n_1,\dots,n_m$ be natural numbers such that $$\sum_{i=1}^mn_ia_i=n,\ \ \ \ \ \ \ \ \sum_{i=1}^mn_ia_i^2\geq (1-\ve)n.$$
Let $A=\{i:a_i\geq c/2\}$. Then $cn/(2-c)\leq\sum_{i\in A}n_i\leq(2-c)n/c$ and $\sum_{i\in A}n_ia_i\geq cn/(2-c)$.
\end{corollary}

\begin{proof}
Strictly speaking, this is a corollary of the proof of Lemma \ref{singularvalues} and not just the statement. We just have to duplicate each $a_i$ $n_i$ times and remark that if $a_i=a_j$ in the proof of Lemma \ref{singularvalues}, then either both of them belong to $A$ or neither does.
\end{proof}

\section{Fourier analysis for matrix-valued functions on finite groups}

In this section, we recall the definition and basic properties of a Fourier transform for matrix-valued functions on general groups. For the convenience of the reader and to make clear our choices of conventions and normalizations, which have not been standardized, we provide proofs of the properties. More details can be found in \cite{Moore}, \cite{Terras} or~\cite{Serre}.

The role of characters in the scalar Abelian case is, as one would expect, played by irreducible representations. Slightly less obvious is how one should generalize a product such as $f(x)\overline{\chi(x)}$. A convenient way turns out to be to regard this as a tensor product of $f(x)$ and a $1\times 1$ matrix with the single entry $\overline{\chi(x)}$. Then when $\chi$ is replaced by a more general irreducible unitary representation $\rho$ we obtain the following definition.

\begin{definition}
Let $G$ be a finite group, let $f:G\rightarrow$M$_n(\C)$ be a matrix-valued function and let $\rho:G\to U(m)$ be an irreducible unitary representation. The \emph{Fourier transform of $f$ at $\rho$} is the $mn\times mn$ matrix
\[\hat f(\rho)=\mathbb E_{x\in G} f(x)\otimes \overline{\rho(x)},\]
where $\overline{\rho(x)}$ is the conjugate matrix of $\rho(x)$, that is, the matrix with entries $\overline{\rho(x)}_{i,j}=\overline{\rho(x)_{i,j}}$.
\end{definition}

Our choice of convention needs a little explaining. In order to have a tidy statement of the convolution identity below, we need the function we tensor with $f(x)$ to be a left representation. This rules out defining $\hat f(\rho)$ to be $\E_xf(x)\otimes\rho(x)^*$, since $\rho^*$ is a right representation. The only left representation that specializes to $\overline{\chi}$ when $\rho$ is a character $\chi$ is the conjugate representation~$\overline{\rho}$. 

If $n=1$, so that $f$ takes scalar values, then we obtain the slightly simpler formula
\[\hat f(\rho)=\E_{x\in G}f(x)\overline{\rho(x)},\]
which is very similar to the definition of the Fourier transform for scalar functions defined on Abelian groups. One way of thinking about the definition for matrix-valued functions is that we are applying the formula for scalar-valued functions pointwise. That is, for each $1\leq i,j\leq n$ we define a scalar valued function $f_{ij}$ by $f_{ij}=f(x)_{ij}$. We then form an $n\times n$ block matrix out of the Fourier transforms $\widehat{f_{ij}(\chi)}$ (which are themselves $m\times m$ matrices).
 
In the next lemma we shall prove five basic properties of this Fourier transform. Almost all of them rely on one fundamental lemma in elementary representation theory. Given an irreducible representation $\rho$, write $n_\rho$ for its dimension. Also, when we write $\sum_\rho$ it is to be understood that we are summing over all irreducible representations of the group $G$ we are talking about. The representation-theoretic lemma is the following well-known result of Schur. We write $\chi_\rho$ for the character associated with $\rho$: that is, $\chi_\rho(x)=\tr(\rho(x))$ for each $x\in G$.

\begin{lemma}
Let $G$ be a finite group. Then $\sum_\rho n_\rho\chi_\rho(x)=|G|$ if $x=e$ and 0 otherwise.
\end{lemma}

\noindent This is an orthogonality statement, since it implies that $|G|^{-1}\sum_\rho n_\rho\rho(xy^{-1})=1$ if $x=y$ and 0 otherwise. In the proofs below, we shall never see powers of $|G|$ appearing, because there will always be an expectation that cancels out any such factors. (Another way of looking at this is to think of taking expectations over $G$ as integrating with respect to Haar measure, and the function that takes the value $|G|$ at the identity and 0 everywhere else as the appropriate delta-function for that integral.)

The one slightly non-obvious statement below is the inversion theorem. In ordinary Fourier analysis, we express a function as a linear combination of characters. In the Abelian case we expressed a matrix-valued function as a sort of linear combination of characters, except that the scalars had become matrices. But now we want to express a function that takes values in M$_n(\C)$ as some kind of combination of $n_\rho n\times n_\rho n$ matrices, where $n_\rho$ is different from representation to representation. Somehow we need to turn these matrices into $n\times n$ matrices. There is a natural way to do this, and it turns out to work well. We define the \emph{partial trace} $\tr_\rho$ of a matrix in M$_n(\C)\otimes$M$_{n_\rho}(\C)$ by defining $\tr_\rho(A\otimes B)$ to be $(\tr (B))A$ and extending linearly. That is, if we regard a matrix in M$_n(\C)\otimes$M$_{n_\rho}(\C)$ as being an $n\times n$ block matrix where each block is an $n_\rho\times n_\rho$ matrix, then we form an $n\times n$ matrix of scalars by taking the trace of each block. 

In the statement of the inversion theorem, we need to multiply each block of $\hat f(\rho)$ on the left by $\overline{\rho(x^{-1})}$. We shall denote the resulting matrix by ${\overline{\rho(x^{-1})}\cdot\hat f(\rho)}$. We think of this as a kind of ``block scalar multiplication" of the matrix. Note that $\overline{\rho(x^{-1})}\cdot\hat f(\rho)$ is just a convenient shorthand for the product of the matrix $I_n\otimes\overline{\rho(x^{-1})}$ with the matrix $\hat f(\rho)$. 

\begin{lemma}\label{fouriernonabelian} The following properties hold for the Fourier transform just defined.
\begin{enumerate}
\item $\mathbb E_x \|f(x)\|_{HS}^2=\sum_{\rho} n_\rho \|\hat f(\rho)\|_{HS}^2$ (\textit{Parseval's identity})
\item $\mathbb E_x \text{\emph{tr}}(f(x)g(x)^*)=\sum_\rho n_\rho\text{\emph{tr}}\left(\hat f(\rho)\hat g(\rho)^*\right)$ (\textit{Parseval's identity 2})
\item $\widehat{f*g} = \hat f\hat g$ (\textit{convolution formula})
\item $f(x)= \sum_\rho n_\rho\text{\emph{tr}}_\rho\Bigl(\overline{\rho(x^{-1})}\cdot\hat{f}(\rho)\Bigr) $(\textit{Fourier inversion formula})
\item $\|f\|_{U^2}^4=\sum_{\rho} n_\rho \|\hat f(\rho)\|_\square^4$ (\textit{$U^2$-norm identity})
\end{enumerate}
\end{lemma}

\begin{proof} 
\begin{enumerate}
\item[]
\item For any two square matrices $A,B$, we have $\tr(A\otimes B)=\tr(A)\tr(B)$. Therefore,
\begin{eqnarray*}
\sum_\rho n_\rho \|\hat f(\rho)\|_{HS}^2 & = & \sum_\rho n_\rho \|\underset{x}{\mathbb E}\ f(x)\otimes \overline{\rho(x)}\|_{HS}^2 \hfill \\
& = & \sum_\rho n_\rho \E_{x,y}\tr((f(x)\otimes\overline{\rho(x)})(f(y)^*\otimes\overline{\rho(y)^*})) \\
& = & \sum_\rho n_\rho \E_{x,y} \mbox{tr}\left(f(x)f(y)^*\right) \tr(\overline{\rho(x)\rho(y)^{*}}) \hfill \\
& = & \sum_\rho n_\rho \E_{x,y}\ \mbox{tr}\left(f(x)f(y)^*\right) \overline{\chi_\rho(xy^{-1})} \hfill \\
& = & \E_{x,y}\ \mbox{tr} \left(f(x)f(y)^*\right) \sum_\rho n_\rho\overline{\chi_\rho(xy^{-1})} \hfill \\
& = & \E_x\  \mbox{tr} \left(f(x)f(x)^*\right)\hfill \\
& = & \E_x\ \|f(x)\|_{HS}^2 \hfill
\end{eqnarray*}

\item 

\begin{eqnarray*}
\sum_\rho n_\rho \mbox{tr}(\hat f(\rho)\hat g(\rho)^*) 
& = &\sum_\rho n_\rho\E_{x,y}\tr((f(x)\otimes\overline{\rho(x)})(g(y)^*\otimes\overline{\rho(y)^*})) \\
& = & \sum_\rho n_\rho \E_{x,y}\ \mbox{tr}\left(f(x)g(y)^*\right)\tr(\overline{\rho(x)\rho(y)^*})\hfill \\
& = & \sum_\rho n_\rho \E_{x,y}\ \mbox{tr}\left(f(x)g(y)^*\right)\overline{\chi_\rho(xy^{-1})}\hfill \\
& = & \E_{x,y}\ \mbox{tr} \left(f(x)g(y)^*\right) \sum_\rho n_\rho\overline{\chi_\rho(xy^{-1})} \hfill \\
& = & \E_x\  \mbox{tr} \left(f(x)g(x)^*\right)\hfill \\
\end{eqnarray*}

\item 
\begin{eqnarray*}
\widehat{f*g}(\rho)&=&\E_z (f*g)(z)\otimes\overline{\rho (z)}\\
&=&\E_{x,y} f(x)g(y)\otimes \overline{\rho(x)\rho(y)}\\
&=&(\E_x f(x)\otimes \overline{\rho(x)})(\E_y g(y)\otimes \overline{\rho(y)})\\
&=&\hat f(\rho)\hat g (\rho)
\end{eqnarray*}

\item 
\begin{eqnarray*}
\sum_\rho n_\rho \tr_\rho\Bigl(\overline{\rho(x^{-1})}\cdot\hat{f}(\rho)\Bigr) 
&=& \sum_\rho n_\rho\tr_\rho\Bigl(\overline{\rho(x^{-1})}\cdot\E_yf(y)\otimes\overline{\rho(y)}\Bigr) \\
&=&\sum_\rho n_\rho\mathbb E_y \mbox{tr}_\rho \left(f(y)\otimes\overline{\rho (x^{-1}y)} \right)\\ 
& = &\mathbb E_y\sum_\rho n_\rho \overline{\chi_\rho (x^{-1}y)}f(y)\\
& = & f(x) \\
\end{eqnarray*}
\item 


\begin{equation*}
\begin{aligned}
\sum_{\rho} n_\rho\|\hat f(\rho)\|_\square^4 & =\sum_\rho n_\rho\tr\left(\hat f(\rho)\hat f(\rho)^*f(\rho)\hat f(\rho)^*\right)\\
& = \E_{x,y,z,w}\ \tr\left(f(x)f(y)^*f(z)f(w)^*\right)\sum_\rho n_\rho \overline{\chi_\rho(xy^{-1}zw^{-1}}) \\
& = \E_{xy^{-1}zw^{-1}=e}\ \mbox{tr}\left(f(x)f(y)^*f(z)f(w)^*\right) \\
& = \|f\|_{U^2}^4. \\
\end{aligned}
\end{equation*}

\end{enumerate}
\end{proof}

\section{Obtaining many approximately invariant subspaces}

We shall now use the Fourier transform just described to obtain a key lemma. But first, let us give an interpretation of the matrix $\hat f(\rho)$ that helps to clarify its meaning.

Because the rows and columns of $\hat f(\rho)$ are indexed by $\{1,\dots,n\}\times\{1,\dots,n_\rho\}$, we can regard it as a linear map defined on $n\times n_\rho$ complex matrices. We have
\begin{eqnarray*}
\hat f(\rho)(A)_{(i,r)}&=&\sum_{(j,s)}\hat f(\rho)_{(i,r),(j,s)}A_{(j,s)}\\
&=&\sum_{j,s}\E_x f(x)_{ij}\overline{\rho(x)}_{rs}A_{js}\\
&=&\E_x\sum_{j,s}f(x)_{ij}A_{js}\rho(x)^*_{sr},\\
\end{eqnarray*}
which gives us that 
\[\hat f(\rho)(A) = \E_x f(x) A\rho(x)^*.\]
This way of thinking about $\hat f(\rho)$ will be essential to the arguments that follow.

\begin{lemma}\label{3conditionsnonabelian} Let $G$ be a finite group and $f:G\rightarrow$M$_n(\C)$ be a function such that $\|f(x)\|_\op\leq 1$ for every $x$ and such that $\|f\|_{U^2}^4\geq cn$. Then  there are irreducible representations $\rho_1,\dots,\rho_m$ such that 
$$\sum_i n_{\rho_i}\in [cn/(2-c), (2-c)n/c],$$
two sequences of matrices $U_1,\dots,U_m$ and $V_1,\dots,V_m$ such that for each $i$ both $U_i$ and $V_i$ are $n\times n_{\rho_i}$ matrices with $\|U_i\|_{HS}^2=\|V_i\|_{HS}^2=n_{\rho_i}$, and 
a partition $\{1,2,\dots,m\}=A_1\cup\dots\cup A_k$, such that the following conditions are satisfied.
\begin{enumerate}
\item $\rho_i=\rho_j$ if and only if $i$ and $j$ belong to the same set $A_r$, and otherwise they are inequivalent.
\item If $i$ and $j$ belong to the same set $A_r$ and are not equal, then $\tr(U_iU_j^*)=\tr(V_iV_j^*)=0$.
\item For each $i$ there exists $\lambda_i\in[0,1]$ such that $\E_xf(x)U_i\rho_i(x)^*=\lambda_iV_i$ and $\E_xf(x)^*V_i\rho_i(x)=\lambda_iU_i$.
\item $\sum_{i=1}^mn_{\rho_i}\lambda_i^2\geq\max\{(c/2)^{1/2},(c/(2-c))^2\}\sum_{i=1}^mn_{\rho_i}$.
\end{enumerate}
\end{lemma}

\begin{proof} By the fifth property in Lemma \ref{fouriernonabelian} we have the bound
\[\sum_{\rho} n_\rho \|\hat f(\rho)\|_\square^4\geq cn.\]
Now $\hat f(\rho)$ is a convex combination of $nn_\rho\times nn_\rho$ matrices of operator norm at most 1, so it too has operator norm at most 1. Let $\lambda_{\rho,1},\dots,\lambda_{\rho,nn_\rho}$ be its singular values. Then $\lambda_{\rho,k}\in[0,1]$ for each $\rho$ and $k$. We also have
\[\sum_{\rho,k}n_\rho\lambda_{\rho,k}^2=\sum_\rho n_\rho\|\hat f(\rho)\|_{HS}^2=n\]
and
\[\sum_{\rho,k}n_\rho\lambda_{\rho,k}^4=\sum_\rho n_\rho\|\hat f(\rho)\|_\square^4=\|f\|_{U^2}^4\geq cn.\]
Hence by Corollary \ref{weightedsingularvalues} if $\lambda_1,\dots,\lambda_m$ are all the singular values that are at least $(c/2)^{1/2}$, and if they correspond to representations $\rho_1,\dots,\rho_m$, then
\[\sum_{i=1}^mn_{\rho_i}\in\bigl[cn/(2-c),(2-c)n/c\bigr]\] and 
\[\sum_{i=1}^mn_{\rho_i}\lambda_i^2\geq\frac{cn}{2-c}\geq(\frac c{2-c})^2\sum_{i=1}^mn_{\rho_i}.\] 
Since each $\lambda_i\in[0,1]$, it follows that 
\[\sum_{i=1}^mn_{\rho_i}\lambda_i\geq(\frac c{2-c})^2\sum_{i=1}^mn_{\rho_i},\]
and since each $\lambda_i\geq(c/2)^{1/2}$ we also have that
 \[\sum_{i=1}^mn_{\rho_i}\lambda_i\geq(c/2)^{1/2}\sum_{i=1}^mn_{\rho_i},\]
Partition the set $\{1,\dots,m\}$ into sets $A_1,\dots,A_k$ according to representations: that is, put $i$ and $j$ in the same set $A_r$ if and only if $\rho_i=\rho_j$.

Now we use the fact that $\hat f(\rho)$ can be regarded as mapping a matrix $A$ to the matrix $\E_xf(x)A\rho(x)^*$. To each  of the singular values just obtained, there correspond two $n\times n_\rho$ matrices $U_i$ and $V_i$ such that $\E_xf(x)U_i\rho_i(x)^*=\lambda_iV_i$ and $\E_xf(x)^*V_i\rho_i(x)=\lambda_iU_i$. We are free to choose a normalization, so we choose it in such a way that $\|U_i\|_{HS}^2=\|V_i\|_{HS}^2=n_{\rho_i}$. Because these matrices come from singular value decompositions, we also have that if $\rho_i=\rho_j$, then $U_i$ and $U_j$ are orthogonal in the Hilbert-Schmidt norm, as are $V_i$ and $V_j$. That is, $\tr(U_iU_j^*)=\tr(V_iV_j^*)=0$. This completes the proof.
\end{proof}

Note that when $c=1-\varepsilon$ for some small $\varepsilon$, we have
\[(c/(2-c))^2=((1-\varepsilon)/(1+\varepsilon))^2\geq 1-4\varepsilon,\] 
which is bigger than $(c/2)^{1/2}$, but for small $c$ the $(c/2)^{1/2}$ bound is better. For brevity, we shall write $\tau(c)$ for $\max\{(c/2)^{1/2},(c/(2-c))^2\}$. 

\section{From approximately invariant subspaces to an approximating representation}

Suppose we have matrices $U_1,\dots,U_m$ and $V_1,\dots,V_m$ and irreducible representations $\rho_1,\dots,\rho_m$ satisfying the conditions of Lemma \ref{3conditionsnonabelian}. Let $t=\sum_{i=1}^mn_{\rho_i}$, let $U$ and $V$ be the $n\times t$ matrices $(U_1|\dots|U_m)$ and $(V_1|\dots|V_m)$ and let $P:G\to U(t)$ be the representation given by the formula 
\[P(x)=\rho_1(x)\oplus\dots\oplus\rho_m(x).\]
Let $\Lambda$ be the matrix $\lambda_1I_{n_{\rho_1}}\oplus\dots\oplus\lambda_mI_{n_{\rho_m}}$.

Then property 3 of Lemma \ref{3conditionsnonabelian} tells us that 
\[\E_xf(x)UP(x)^*=(\lambda_1V_1|\dots|\lambda_mV_m)=\Lambda V,\] 
from which it follows that 
\begin{eqnarray*}
\E_x\langle f(x),VP(x)U^*\rangle&=&\langle\E_xf(x)UP(x)^*,V\rangle\\
&=&\langle\Lambda V,V\rangle\\
&=&\sum_i\lambda_i\|V_i\|_{HS}^2\\
&=&\sum_{i=1}^mn_{\rho_i}\lambda_i.\\
\end{eqnarray*}
By property 4 of Lemma \ref{3conditionsnonabelian} and the fact that each $\lambda_i$ is between 0 and 1, this is at least $\tau(c)t$, where $\tau$ is the function defined at the end of the previous section. Here the inner product is as usual the matrix inner product $\langle A,B\rangle=\tr(AB^*)=\sum_{ij}A_{ij}\overline{B_{ij}}$. 

The next lemma is the main driver of the rest of the proof, and the place where we use the orthogonality properties of representations and their matrix elements. (In the Abelian case, this is just the orthogonality of the characters.)

\begin{lemma} \label{rank1}
Let $\rho_1,\dots,\rho_m$ be irreducible representations with $\sum_in_{\rho_i}=t$ such that any two are either equal or inequivalent, let $P:G\to U(t)$ be the representation $\rho_1\oplus\dots\oplus\rho_m$, for each $i$ let $U_i$ be an $n\times n_{\rho_i}$ matrix with columns of $\ell_2$-norm at most 1, let $U=(U(1)|\dots|U(m))$, and let $a\in\C^n$ and $b\in\C^t$ be unit vectors. Suppose that $\tr(U(i)^*U(j))=0$ whenever $\rho_i=\rho_j$ but $i\ne j$. Then 
\[\E_x\|(a\otimes b)P(x)U^*\|_\nuc^2\leq 1.\]
\end{lemma}

\begin{proof}
For any $n\times t$ matrix $T$ and any $i\leq n$, $j\leq m$, we have
\[((a\otimes b)T^*)_{ij}=\sum_ka_ib_k\overline{T_{kj}}=a_i(T^*b)_j=(a\otimes T^*b)_{ij}.\]
Therefore, $(a\otimes b)T^*=a\otimes T^*b$, the nuclear norm of which is $\|T^*b\|_2$. Therefore, what we are trying to bound above is equal to $\E_x\|UP(x)^*b\|_2^2$.



Let $b=b(1)+\dots+b(m)$, where $b(i)$ is the part of $b$ that is acted on by $U(i)\rho_i(x)^*$.Then 
\[UP(x)^*b=\sum_pU(p)\rho_p(x)^*b(p).\]
Therefore,
\begin{eqnarray*}\E_x\|UP(x)^*b\|_2^2&=&\E_x\sum_i|\sum_p(U(p)\rho_p(x)^*b(p))_i|^2\\
&=&\E_x\sum_i\sum_{p,q}\sum_{j,k,r,s}U(p)_{ij}\overline{\rho_p(x)_{kj}}b(p)_k\overline{U(q)_{ir}}\rho_q(x)_{sr}\overline{b(q)_s}.\\
\end{eqnarray*}
By the orthogonality of matrix elements, the expectation over $x$ gives us zero unless $\rho_p=\rho_q$, $j=r$ and $k=s$. If all three of these equalities holds, it gives us $n_{\rho_p}^{-1}$. Therefore, writing $p\sim q$ to mean that $\rho_p=\rho_q$, this expression simplifies to
\[\sum_i\sum_{p\sim q}n_{\rho_p}^{-1}\sum_{j,k}U(p)_{ij}\overline{U(q)_{ij}}|b(p)_k|^2\]
But by hypothesis, when $\rho_p=\rho_q$ and $p\ne q$ we have that 
\[\sum_{i,j}U(p)_{ij}\overline{U(q)_{ij}}=\tr(U(p)U(q)^*)=0,\] 
so this simplifies further to
\[\sum_i\sum_pn_{\rho_p}^{-1}\sum_j|U(p)_{ij}|^2\sum_k|b(p)_k|^2.\]
For each fixed $p$ we have $\sum_{i,j}|U(p)_{ij}|^2\leq n_{\rho_p}$, because of the upper bound on the column sizes of $U(p)$, so this is at most $\sum_p\sum_k|b(p)_k|^2$, which equals $\sum_k|b_k|^2$, which equals $1$.
\end{proof}

Note that we get equality in the inequality above if and only if the columns of $U$ all have unit length.

\begin{corollary} \label{nuclearlemma}
Let $U$ and $P$ satisfy the conclusion of Lemma \ref{rank1} and let $V$ be another $n\times t$ matrix. Then 
\[\E_x\|VP(x)U^*\|_\nuc\leq\|V\|_\nuc.\]
\end{corollary}

\begin{proof}
Let $V=\sum_i\lambda_ia_i\otimes b_i$ with $a_i$ and $b_i$ unit vectors and each $\lambda_i$ a non-negative real number and with $\sum_i\lambda_i=\|V\|_\nuc$. Then
\begin{eqnarray*}
\E_x\|VP(x)U^*\|_\nuc&\leq&\sum_i\lambda_i\E_x\|(a_i\otimes b_i)P(x)U^*\|_\nuc\\
&\leq&\sum_i\lambda_i(\E_x\|(a_i\otimes b_i)P(x)U^*\|_\nuc^2)^{1/2}\\
&\leq&\sum_i\lambda_i,\\
\end{eqnarray*}
where the last inequality follows from our hypothesis. This gives us an upper bound of $\|V\|_\nuc$ as required.
\end{proof}

\begin{lemma} \label{nuclearapprox}
Let $A$ be a matrix with $\|A\|_{HS}^2\leq m$. Then for every $C$ there exist matrices $A'$ and $A''$ with $A'+A''=A$ and
\[C^{-1}\|A'\|_\op+Cm^{-1}\|A''\|_\nuc\leq 1.\]
\end{lemma}

\begin{proof}
Suppose that we cannot find such a pair of matrices. Then by the Hahn-Banach theorem and the duality of the nuclear and operator norms, there exists a linear functional $\phi$ such that $\tr(A\phi^*)>1$, $\|\phi\|_\nuc\leq C^{-1}$ and $\|\phi\|_\op\leq C/m$. But then, by Lemma \ref{nuclearop},
\[\|\phi\|_{HS}^2=\tr(\phi\phi^*)\leq\|\phi\|_\nuc\|\phi\|_\op\leq1/m\]
from which it follows that $\|A\|_{HS}>m^{1/2}$, which we know not to be the case.
\end{proof}

Recall that we define a partial unitary matrix to be one that can be extended to a unitary matrix by the addition of some rows or columns. Equivalently, it is a matrix with all its singular values equal to 1. (This definition was given just after Lemma \ref{partialunitary}.)

From this point onwards in the proof, we care less about the internal structure of our matrices and representations, so we shall let $m$ be the number of columns of $U$ and $V$ rather than the number of blocks. Thus, the role played up to now by $t$ will be played by $m$.

\begin{corollary} \label{vpartialunitary}
Let $G$ be a finite group and let $f:G\to$M$_n(\C)$ be a function such that $\|f(x)\|_\op\leq 1$ for every $x\in G$. Let $P:G\to U(n)$, let $U$ be an $n\times m$ matrix, and suppose that $\E_x\|WP(x)U^*\|_\nuc\leq\|W\|_\nuc$ for every $n\times m$ matrix $W$. Let $V$ be an $n\times m$ matrix and suppose that $\|V\|_{HS}^2\leq m$. Suppose also that $\theta>0$ and that
\[|\E_x\langle f(x),VP(x)U^*\rangle|\geq \theta m.\] 
Then there is a partial unitary matrix $V'$ such that 
\[|\E_x\langle f(x),V'P(x)U^*\rangle|\geq \theta^2m.\]
\end{corollary}

\begin{proof}
We shall prove that for every $a<\theta$ there is a matrix $V'$ satisfying the conclusion with $a^2$ replacing $\theta^2$. This is enough, by compactness.

By Lemma \ref{nuclearapprox} we can find $V'$ with 
\[a\|V'\|_\op+a^{-1}m^{-1}\|V-V'\|_\nuc\leq 1,\]
which implies that
\[a^2m\|V'\|_\op+\|V-V'\|_\nuc\leq am.\] 
By hypothesis on $U$ and $P$ we have that $\E_x\|(V-V')P(x)U^*\|_\nuc\leq\|V-V'\|_\nuc$. Since $\|f(x)\|_\op\leq 1$ for every $x$, it follows that
\[|\E_x\langle f(x),(V-V')P(x)U\rangle|\leq\|V-V'\|_\nuc,\]
so
\[a^2m\|V'\|_\op+|\E_x\langle f(x),(V-V')P(x)U^*\rangle|\leq am,\]
and therefore
\[|\E_x\langle f(x),(V-V')P(x)U^*\rangle|\leq am-a^2m\|V'\|_\op.\]
It follows that $V'$ cannot be the zero matrix, since then we would contradict our main hypothesis. 

Using that hypothesis, and the inequality above, we may deduce that
\[|\E_x\langle f(x),V'P(x)U^*\rangle|\geq a^2m\|V'\|_\op.\]

We now need to make $V'$ a partial unitary matrix. Since $V'\ne 0$, we can normalize it so that $\|V'\|_\op=1$. Then $V'$ is a convex combination of partial unitary matrices. If $V'=\sum_ic_iW_i$ with $c_i\geq 0$ and $\sum_ic_i=1$, then 
\[\sum_ic_i|\E_x\langle f(x),W_iP(x)U^*\rangle|\geq a^2m,\]
from which it follows that there exists $i$ such that
\[|\E_x\langle f(x),W_iP(x)U^*\rangle|\geq a^2m,\]
and we are done. 
\end{proof}

\begin{corollary} \label{bothpartialunitary}
Let $G$ be a finite group, let $f:G\to$M$_n(\C)$ be a function such that $\|f(x)\|_\op\leq 1$ for every $x\in G$ and let $P:G\to U(n)$ be a unitary representation. Suppose that $U$ is an $n\times m$ matrix such that $\E_x\|WP(x)U^*\|_\nuc\leq\|W\|_\nuc$ for every $n\times m$ matrix $W$, and that $V$ is an $n\times m$ matrix such that $\|V\|_{HS}^2\leq m$. Let $\theta>0$ and suppose also that 
\[|\E_x\langle f(x),VP(x)U^*\rangle|\geq \theta m.\] 
Then there are partial unitary matrices $U'$ and $V'$ such that 
\[|\E_x\langle f(x),V'P(x)U'^*\rangle|\geq \theta^4m.\]
\end{corollary}

\begin{proof}
By Corollary \ref{vpartialunitary} we can find a partial unitary matrix $V'$ such that 
\[|\E_x\langle f(x),V'P(x)U^*\rangle|\geq \theta^2m.\] 
This is equivalent to the statement that 
\[|\E_x\langle f(x)^*,UP(x)^*V'^*\rangle|\geq \theta^2m.\] 
Now if $W$ is any $n\times m$ matrix and $x\in G$, then 
\[\|WP(x)V'^*\|_\nuc\leq\|W\|_\nuc\|P(x)\|_\op\|V'^*\|_\op=\|W\|_\nuc,\]
since $P(x)$ is unitary and $V'$ is partially unitary. Therefore, the hypotheses of Corollary \ref{vpartialunitary} hold for $P^*$ and $V$, with $\theta^2$ replacing $\theta$. It follows that there exists a partial unitary matrix $U'$ such that 
\[|\E_x\langle f(x)^*,U'P(x)^*V'^*\rangle|\geq \theta^4m,\] 
which is equivalent to the statement that
\[|\E_x\langle f(x),V'P(x)U'^*\rangle|\geq \theta^4m,\]
which proves the result.
\end{proof}

We have now more or less proved our promised inverse theorem for the matrix-valued $U^2$ norm. Recall that $\tau(c)=\max\{(c/(2-c))^2,(c/2)^{1/2}\}$, so $\tau(c)^4$ is $O(c^2)$ when $c$ is small and $1-O(\varepsilon)$ when $c=1-\varepsilon$. Later in the paper, we shall use the more precise estimates $(1-\varepsilon)/(1+\varepsilon)\geq 1-2\varepsilon$ and $\tau(1-\varepsilon)^4\geq 1-16\varepsilon$, both of which are straightforward to check.

\begin{theorem} \label{inversetheorem}
Let $G$ be a finite group, let $c>0$ and let $f:G\to$M$_n(\C)$ be a function such that $\|f(x)\|_\op\leq 1$ for every $x\in G$ and $\|f\|_{U^2}^4\geq cn$. Then there exists $m\in[cn/(2-c),(2-c)n/c]$, $n\times m$ partial unitary matrices $U$ and $V$, and a unitary representation $P:G\to U(m)$ such that 
\[|\E_x\langle f(x), VP(x)U^*\rangle|\geq \tau(c)^4m.\] 
\end{theorem}

\begin{proof}
By Lemma \ref{3conditionsnonabelian} and the remark at the beginning of this section, there exist $n\times m$ matrices $U_0$ and $V_0$ such that $m\in [cn/4,4n/c]$ and 
\[|\E_x\langle f(x),V_0P(x)U_0^*\rangle|\geq \tau(c)m.\] 
Moreover, $P$ and $U$ satisfy the conditions of Lemma \ref{rank1}, and hence the conclusion of Corollary \ref{nuclearlemma}, which is the hypothesis needed for Corollary \ref{bothpartialunitary}. Furthermore, $\|V\|_{HS}^2=m$ (this is given to us by Lemma \ref{3conditionsnonabelian}, with $m$ replacing $t$). The result then follows from Corollary \ref{bothpartialunitary} with $\theta=\tau(c)$.
\end{proof}

\section{Obtaining stability theorems from the inverse theorem}

We shall now prove stability theorems for approximate representations in the (normalized) Schatten $p$-norms for $1\leq p\leq 2$. Given a matrix $A\in$M$_n(\C)$ with singular values $\lambda_1,\dots,\lambda_n$, its \emph{Schatten $p$-norm} $\|A\|_p$ is defined to be $(\sum_{i=1}^n\lambda_i^p)^{1/p}$. We define its \emph{normalized Schatten $p$-norm} to be $n^{-1/p}\|A\|_p=(\E_i\lambda_i^p)^{1/p}$, which we denote by $\|A\|_p'$. Of particular interest are the cases $p=1$ and $p=2$, which give us the normalized nuclear and Hilbert-Schmidt norms, respectively. Our aim is to prove for each of these norms that if $G$ is a finite group and $f:G\to U(n)$ is an approximate representation then $f$ is approximated, in a suitable sense, by a genuine representation. 

The following lemma (which is standard, but usually stated for square matrices) will be needed later. 

\begin{lemma} \label{multiplybysmallop}
Let $A$ be an $m\times m$ complex matrix and let $B$ and $C$ be $n\times m$ complex matrices. Then $\|BAC^*\|_p\leq\|A\|_p\|B\|_\op\|C\|_\op$. 
\end{lemma}

\begin{proof}
The basic fact we use is that if $A\in$M$_n(\C)$ and $B$ is an $n\times n$ unitary matrix, then $\|AB\|_p=\|BA\|_p=\|A\|_p$. Since an $n\times n$ matrix with operator norm at most 1 is a convex combination of unitary matrices, we get the result when $m=n$. If $m<n$, then add $n-m$ columns of zeros to $B$ and $C$ to create matrices $B'$ and $C'$ and let $\|A'\|=A\oplus 0_{n-m}$. Then $B'A'C'^*=BAC^*\oplus 0_{n-m}$, so we have $\|A'\|_p=\|A\|_p$, $\|B'\|_\op=\|B\|_\op$, $\|C'\|_\op=\|C\|_\op$ and $\|B'AC'^*\|_p=\|BAC\|_p$. But $\|B'A'C'^*\|_p\leq\|A'\|_p\|B'\|_\op\|C'\|_\op$ by the result for square matrices, so the lemma follows.

If $m>n$, then add $m-n$ rows of zeros to $B$ and $C$ to create matrices $B'$ and $C'$. This time $B'AC'^*=BAC^*\oplus 0_{m-n}$ and again $\|B'\|_\op=\|B\|_\op$ and $\|C'\|_\op=\|C\|_\op$, so the result again follows from the result for square matrices. 
\end{proof}

Now let us give a precise definition of ``approximate representation". We say that a norm $\|.\|$ on M$_n(\C)$ is \emph{invariant} if $\|UAV\|=\|A\|$ for any two matrices $U,V\in U(n)$. Since multiplying on either side by a unitary matrix preserves singular values, the Schatten $p$-norms are invariant for every $p$. 

\begin{definition}
Let $G$ be a group, let $\ve>0$, let $n$ be a positive integer, let $\|.\|$ be an invariant norm on M$_n(\C)$ and let $f:G\to U(n)$. Then $f$ is an $\ve$-\emph{representation with respect to $\|.\|$} if $\|f(x)f(y)-f(xy)\|\leq\ve$ for every $x,y\in G$. It is an \emph{affine $\ve$-representation (or Freiman $\ve$-homomorphism) with respect to} $\|.\|$ if $\|f(x)f(y)^*f(z)f(w)^*-I\|\leq\ve$ for every $x,y,z,w\in G$ for which ${xy^{-1}zw^{-1}=e}$. 
\end{definition}

We are ready to begin the proof of our stability theorem. In fact, we prove two theorems, one for affine $\ve$-representations and one for $\ve$-representations. We shall deduce the latter from the former. 

We begin with a technical lemma.

\begin{lemma} \label{distancefromid}
Let $\ve>0$, let $1\leq p\leq 2$, and let $A\in U(n)$ be a matrix such that $\|A-I_n\|_p'\leq\ve$. Then $\Re\,\tr'(A)\geq 1-2^{1-p}\ve^p$.
\end{lemma}

\begin{proof}
Let the singular values of $A-I_n$ be $\lambda_1,\dots,\lambda_n$. Then the maximum singular value is at most 2, so 
\[\|A-I_n\|_{hs}^2=\E_i\lambda_i^2\leq 2^{2-p}\E_i\lambda_i^p\leq 2^{2-p}\ve^p.\]
But we also have that
\[\|A-I_n\|_{hs}^2=\tr'((A-I_n)(A^*-I_n))=2-2\Re\,\tr'(A).\]
The result follows.
\end{proof}

We shall now use the inverse theorem to obtain a representation that approximates $f$ on average. Having done that, we shall show that the approximation is in fact uniform (in the sense that it holds for every $x\in G$). 

\begin{lemma} \label{averagepstability}
Let $1\leq p\leq 2$, let $\ve>0$ and suppose that $2^{1-p}\ve^p<1/4$. Let $G$ be a finite group and let $f:G\to U(n)$ be an affine $\ve$-representation with respect to the normalized Schatten $p$-norm $\|.\|_p'$. Then there exist $m\in[(1-2^{2-p}\ve^p)n,(1-2^{2-p}\ve^p)^{-1}n]$, $n\times m$ partial unitary matrices $U$ and $V$, and a unitary representation $P:G\to U(m)$ such that 
\[\|\E_xf(x)UP(x)^*V^*-I_n\|_p'\leq C_p\ve,\]
where $C_p=(2^{5-p}+2^{2-p})^{1/p}$.
\end{lemma}

\begin{proof}
If $xy^{-1}zw^{-1}=e$, then by hypothesis 
\[\|f(x)f(y)^*f(z)f(w)^*-I_n\|_p'\leq\ve.\] 
By Lemma \ref{distancefromid}, this implies that $\tr(f(x)f(y)^*f(z)f(w)^*)\geq (1-2^{1-p}\ve^p)n$. It follows that $\|f\|_{U^2}^4\geq(1-2^{1-p}\ve^p)n$. By Theorem \ref{inversetheorem} and the preceding remarks about bounds, we can find $m$ in the stated range, partial unitary matrices $U$ and $V$, and a representation $P:G\to U(m)$ such that 
\[|\E_x\tr(f(x)UP(x)^*V^*)|\geq(1-2^{5-p}\ve^p)m.\]
It follows that the sum of the singular values of $\E_xf(x)UP(x)^*V^*$ is at least $(1-2^{5-p}\ve^p)m$, which in turn is at least $(1-(2^{5-p}+2^{2-p})\ve^p)n=(1-C_p^p\ve^p)\ve n$. 

Let $W$ be the matrix obtained from $\E_xf(x)UP(x)^*V^*$ by replacing all its singular values by 1. Then $W\in U(n)$ and 
$\|\E_xf(x)UP(x)^*V^*-W\|_1'\leq C_p^p\ve^p$. Since the singular values of $\E_xf(x)UP(x)^*V^*-W$ lie between 0 and 1, this implies that $\|\E_xf(x)UP(x)^*V^*W^*-I_n\|_p'\leq C_p\ve$. This implies the result, since $WV$ is a partial unitary $n\times m$ matrix.
\end{proof}

There is now an obvious candidate for an affine representation that approximates $f$: the map $\rho$ defined by the formula $\rho(x)=VP(x)U^*$ for each $x$, where $U$, $V$ and $P$ are given by Lemma \ref{averagepstability}. Actually, we shall end up choosing a translate of this map. This will certainly be a map of the required type, but to prove that it approximates $f$ everywhere, we shall need to know that maps of the above type are approximate affine representations. First, we give a name to them.

\begin{definition} Let $G$ be a finite group and let $n$ and $m$ be positive integers. An $(n,m)$-\emph{partial affine representation} is a function $\rho:G\to$ M$_n(\C)$ of the form $\rho(x)=VP(x)U^*$, where $U$ and $V$ are $n\times m$ partial unitary matrices and $P:G\to U(m)$ is a unitary representation. If $V=U$, then $\rho$ is an $(n,m)$-\emph{partial representation}.
\end{definition}




The precise fact we need is not quite that partial representations are approximate representations, but that is also true (and can be proved by the same method).

\begin{lemma} \label{flexiblepartialrep}
Let $n$ and $m$ be positive integers and let $\rho$ be an $(n,m)$-partial affine representation with respect to the normalized Schatten $p$-norm $\|.\|_p'$. Then for any $x,y,z,w\in G$ with $xy^{-1}zw^{-1}=e$, we have that
\[\|\rho(x)\rho(y)^*\rho(z)-\rho(w)\|_p'\leq\eta,\]
where $\eta=0$ if $m\leq n$ and $\eta=2((m-n)/n)^{1/p}$ if $m\geq n$.
\end{lemma}

\begin{proof}
Let $U$ and $V$ be $n\times m$ partial unitary matrices and let $P:G\to U(m)$ be a unitary representation such that $\rho(x)=VP(x)U^*$ for every $x$. Then $xy^{-1}z=w$, so $P(x)P(y)^*P(z)=P(w)$, and therefore
\begin{align*}
\rho(x)\rho(y)^*\rho(z)&-\rho(w)=VP(x)U^*UP(y)^*V^*VP(z)U^*-VP(x)P(y)^*P(z)U^*\\
&=VP(x)U^*UP(y)^*V^*VP(z)U^*-VP(x)P(y)^*V^*VP(z)U^*\\
&+VP(x)P(y)^*V^*VP(z)U^*-VP(x)P(y)^*P(z)U^*.\\
\end{align*}
If $m\leq n$, then $U^*U=V^*V=I_m$ and we can see from the first line that this is zero. If $m\geq n$, then by Lemma \ref{multiplybysmallop} and the triangle inequality it is at most $\|U^*U-I_m\|_p'+\|V^*V-I_m\|_p'$, which is at most $2((m-n)/n)^{1/p}$.
\end{proof}

We shall also need to know that if $\rho$ is a partial representation, then $\rho(e)$ is close to a unitary matrix. (The same is true for every $\rho(x)$, but we do not explicitly need this.) Since $\rho(e)$ will be of the form $VU^*$ for $n\times m$ partial unitary matrices $U$ and $V$, the next lemma tells us what we want.

\begin{lemma} \label{almostunitary}
Let $1\leq p\leq\infty$, let $n$ and $m$ be positive integers, and let $U$ and $V$ be $n\times m$ partial unitary matrices. Then there exists a matrix $W\in U(n)$ such that $\|VU^*-W\|_p'\leq(|m-n|/n)^{1/p}$.
\end{lemma}

\begin{proof}
Suppose first that $m\leq n$. If the columns of $U$ are $u_1,\dots,u_m$ and the columns of $V$ are $v_1,\dots,v_m$, then the $u_i$ and $v_i$ are orthonormal sequences, and $VU^*=\sum_iv_i\otimes u_i$. Therefore, $VU^*$ has $m$ singular values equal to 1 and $n-m$ singular values equal to 0. Extending the sequences $(u_i)$ and $(v_i)$ to orthonormal bases of $\C^n$ then gives rise to a unitary matrix $W$ with $\|VU^*-W\|_p=(n-m)^{1/p}$ and hence $\|VU^*-W\|_p'=((n-m)/n)^{1/p}$.

If $m\geq n$ then we need a slightly more complicated argument. Note first that $U^*U$ and $V^*V$ are in this case orthogonal projections on $\C^m$ of rank $n$. It follows that $\|U^*U-I_m\|_1=m-n$. From this and Lemma \ref{multiplybysmallop} it follows that
\[|\tr(VU^*UV^*)-\tr(VV^*)|\leq\|VU^*UV^*-VV^*\|_1\leq m-n.\]
But $VV^*=I_n$, so it follows that $\tr(VU^*UV^*)\geq 2n-m$.

Let $\lambda_1,\dots,\lambda_n$ be the singular values of $VU^*$. Then we have just proved that 
\[\sum_i(1-\lambda_i^2)\leq n-(2n-m)=m-n.\]
It follows that 
\[\sum_i(1-\lambda_i)^p\leq\sum_i(1-\lambda_i)\leq\sum_i(1-\lambda_i^2)\leq m-n.\]
Let $W$ be the unitary matrix obtained from $VU^*$ by replacing all its singular values by 1. Then the above estimate gives us that $\|VU^*-W\|_p\leq (m-n)^{1/p}$, which proves the result.
\end{proof}

For the next lemma it will be convenient to adopt the notation $A\approx_\theta B$ to mean that $\|A-B\|_p\leq\theta$. Note that the triangle inequality translates into the approximate transitivity property that if $A\approx_\theta B$ and $B\approx_\eta C$, then $A\approx_{\theta+\eta}C$. Also, the relation $\approx_\theta$ is symmetric, and if $A\approx_\theta B$ and $\|C\|_\op\leq 1$, then $AC\approx_\theta BC$ and $CA\approx_\theta CB$. (This last statement follows from Lemma \ref{multiplybysmallop}.)




\begin{lemma} \label{exactrepn}
Let $G$ be a finite group, let $\ve,\eta>0$, let $n$ and $m$ be positive integers, let $1\leq p\leq 2$, let $f:G\to U(n)$ be an affine $\ve$-representation with respect to $\|.\|_p'$ and let $\rho:G\to$M$_n(\C)$ be an $(n,m)$-partial representation with respect to $\|.\|_p'$. Let $\d=(|m-n|/n)^{1/p}$. Suppose that 
\[\|I_n-\E_xf(x)\rho(x)^*\|_p'\leq\eta.\]
Then there exists a unitary matrix $W$ such that $\|f(x)-\rho(x)W^*\|_p'\leq\g$ for every $x$, where $\g=\ve+3\d+2\eta$ if $m\geq n$ and $\g=\ve+\d+2\eta$ if $m\leq n$.
\end{lemma} 

\begin{proof}
Our hypothesis states that $I_n\approx_\eta\E_xf(x)\rho(x)^*$, and we also know that $\|f(x)f(e)^*\|_\op\leq 1$. Therefore, 
\[f(x)f(e)^*\approx_\eta f(x)f(e)^*\E_yf(y)\rho(y)^*\]
for every $x$. Since $f$ is an affine $\ve$-representation and takes unitary values, $f(x)f(e)^*f(y)\approx_\ve f(xy)$ for every $x$ and $y$, so by Lemma \ref{multiplybysmallop} and the triangle inequality,
\[f(x)f(e)^*\E_yf(y)\rho(y)^*\approx_\ve\E_yf(xy)\rho(y)^*\]
for every $x$.

By Lemma \ref{flexiblepartialrep}, $\rho(ex^{-1}y)\approx_{2\d}\rho(e)\rho(x)^*\rho(y)$ for every $x$ and $y$ if $m\geq n$, while the two are equal if $m\leq n$. By that and the invariance of $\|.\|_p'$ under taking adjoints, it follows that 
\[\E_yf(xy)\rho(y)^*=\E_yf(y)\rho(ex^{-1}y)^*\approx_{2\d}(\E_yf(y)\rho(y)^*)\rho(x)\rho(e)^*\]
for every $x$ if $m\geq n$, while the two sides are equal if $m\leq n$.

By hypothesis $\E_yf(y)\rho(y)^*\approx_\eta I_n$, from which it follows that
\[\E_yf(y)\rho(y^*)\rho(x)\rho(e)^*\approx_\eta\rho(x)\rho(e)^*.\]
Putting all these estimates together, we deduce that
\[f(x)f(e)^*\approx_{\theta}\rho(x)\rho(e)^*\]
for every $x$, where $\theta=\ve+2\d+2\eta$ if $m\leq n$ and $\theta=\ve+2\eta$ if $m\leq n$. It follows that
\[f(x)=f(x)f(e)^*f(e)\approx_{\theta}\rho(x)\rho(e)^*f(e)\]
for every $x$. 

By Lemma \ref{almostunitary} there is a unitary matrix $W$ such that $\|\rho(e)-W\|_p'\leq\d$. Since $f(e)$ is unitary, it follows that there is a unitary matrix $W$ such that $\|f(e)^*\rho(e)-W\|_p'\leq\d$. Then
\[f(x)\approx_{\theta+\d}\rho(x)W^*\]
for every $x$, which proves the lemma.
\end{proof} 

We are now ready to prove a stability theorem for affine $\ve$-representations. 

\begin{theorem} \label{affinepcase}
Let $G$ be a finite group, let $n$ be a positive integer, let $1\leq p\leq 2$, let $0<\ve\leq 1/4$ and let $f:G\to U(n)$ be an affine $\ve$-representation with respect to $\|.\|_p'$. Then there exists $m\in[(1-2^{2-p}\ve^p)n,(1-2^{2-p}\ve^p)^{-1}n]$ and an $(n,m)$-partial affine representation $\rho$ such that 
\[\|f(x)-\rho(x)\|_p'\leq(1+3.2^{3/p-1}+2C_p)\ve\] 
for every $x\in G$, where $C_p=(2^{5-p}+2^{2-p})^{1/p}$. 
\end{theorem}

\begin{proof}
By Lemma \ref{averagepstability} there exist $m\in[(1-2^{2-p}\ve^p)n,(1-2^{2-p}\ve^p)^{-1}n]$ and an $(n,m)$ partial representation $\sigma$ such that $\|\E_xf(x)\rho(x)^*-I_n\|_p'\leq C_p\ve$. Since $\ve\leq 1/4$, $(1-2^{2-p}\ve^p)^{-1}\leq 1+2^{3-p}\ve^p$. 

We now apply Lemma \ref{exactrepn}. We get that $\d\leq 2^{3/p-1}\ve$ and $\eta=C_p\ve$. So it gives us a unitary matrix $W\in U(n)$ such that \[
\|f(x)-\sigma(x)W^*\|_p'\leq \ve+3\d+2\eta\leq(1+3.2^{3/p-1}+2C_p)\ve\] 
for every $x\in G$. Set $\rho(x)=\sigma(x)W^*$ for every $x$. Then $\rho$ is an $(n,m)$-partial affine representation, so we are done.
\end{proof}

We have tried not to throw too much away in calculating the above bound, so it is a little unpleasant. However, when $p=1$ we have that $1+3.2^{3/p-1}+2C_p=49$, and when $p=2$ it is less than $12$. Also, the bound is decreasing in $p$, so a bound of $49\ve$ is valid for all $p\in [1,2]$.

Now let us deduce a stability theorem for $\ve$-representations. We begin with two lemmas that relate $\ve$-representations to affine $\ve$-representations.

\begin{lemma} \label{repisaffinerep}
Let $G$ be a group, let $\ve>0$, let $n$ be a positive integer, let $1\leq p\leq\infty$ and let $f:G\to U(n)$ be an $\ve$-representation with respect to $\|.\|_p'$. Then $f$ is an affine $2\ve$-representation with respect to $\|.\|_p'$.
\end{lemma}

\begin{proof}
Let $x,y,z,w\in G$ be such that $xy^{-1}zw^{-1}=e$. Then
\begin{eqnarray*}
&&\|f(x)f(y)^*f(z)f(w)^*-I\|_p'\\
&\leq&\|f(x)f(y)^*f(z)f(w)^*-f(x)f(y)^*f(zw^{-1})\|_p'+\|f(x)f(y)^*f(zw^{-1})-I\|_p'\\
&=&\|f(z)f(w)^*-f(zw^{-1})\|_p'+\|f(y)^*-f(x)^*f(yx^{-1})^*\|_p'\\
&=&\|f(z)f(w)^*-f(zw^{-1})\|_p'+\|f(y)-f(yx^{-1})f(x)\|_p'\\
&\leq&2\ve,\\
\end{eqnarray*}
as required.
\end{proof}

In the other direction, we need the following result.

\begin{lemma} \label{partialaffinetorep}
Let $G$ be a finite group, let $n$ and $m$ be positive integers, and let $\sigma:G\to$M$_n(\C)$ be an $(n,m)$-partial affine representation. Let $1\leq p\leq\infty$, let $\d=(|m-n|/n)^{1/p}$, and let $\rho':G\to$M$_n(\C)$ be defined by $\rho'(x)=\sigma(x)\sigma(e)^*$. Then $\rho'$ is an $(n,m)$-partial representation if $m\leq n$, and if $m\geq n$ then there is an $(n,m)$-partial representation $\rho:G\to U(n)$ such that $\|\rho(x)-\rho'(x)\|_p'\leq\d$ for every $x$. 
\end{lemma}

\begin{proof}
Let $\sigma(x)=VP(x)U^*$ for each $x$, where $U$ and $V$ are $n\times m$ partial unitary matrices and $P:G\to U(m)$ is a representation. Then $\rho'(x)=VP(x)U^*UV^*$ for each $x$.

If $m\leq n$, then $U^*U=I_m$, so $\rho'(x)=VP(x)V^*$, and is therefore an $(n,m)$-partial representation, by definition. If $m\geq n$, then $U^*U$ is an orthogonal projection of rank $n$, so $\|U^*U-I_m\|_p'=\d$, from which it follows that $\|\rho'(x)-VP(x)V^*\|_p'\leq\d$ for every $x$. So we may take $\rho(x)$ to be $VP(x)V^*$. 
\end{proof}

Now we can prove a stability result for $\ve$-representations with respect to~$\|.\|_p'$.

\begin{theorem} \label{reppcase}
Let $G$ be a finite Abelian group, let $n$ be a positive integer, let $0<\ve\leq 1/16$ and let $f:G\to U(n)$ be a $\ve$-representation with respect to $\|.\|_p'$. Then there exists $m\in[(1-4\ve^p)n,(1-4\ve^p)^{-1}n]$ and an $(n,m)$-partial representation $\rho$ such that 
\[\|f(x)-\rho(x)\|_p'\leq(1+2D_p+8^{1/p})\ve\] 
for every $x\in G$, where $D_p=1+3.2^{3/p-1}+2(2^{5-p}+2^{2-p})^{1/p}$.
\end{theorem}

\begin{proof}
By Lemma \ref{repisaffinerep} $f$ is an affine $2\ve$-representation with respect to $\|.\|_p'$. Therefore, by Theorem \ref{affinepcase}, there exists $m\in[(1-4\ve^p)n,(1-4\ve^p)^{-1}n]$ and an $(n,m)$-partial affine representation $\sigma$ such that $\|f(x)-\sigma(x)\|_p'\leq 2D_p\ve$ for every $x$. Let $\rho'(x)=\sigma(x)\sigma(e)^*$ for each $x$. Then
\begin{eqnarray*}
\|f(x)-\rho'(x)\|_p'&=&\|f(x)-\sigma(x)\sigma(e)^*\|_p'\\
&\leq&\|f(x)(I_n-f(e)^*)\|_p'+\|f(x)(f(e)^*-\sigma(e)^*)\|_p'\\
&&\ \ \ \ \ \ \ \ \ \ \ \ \ \ \ \ \ \ +\|(f(x)-\sigma(x))\sigma(e)^*\|_p'.\\
\end{eqnarray*}
But 
\[\|f(e)-I_n\|_p'=\|f(e)f(e)-f(e)\|_p'\leq\ve\]
by the definition of an $\ve$-representation, while the other two terms are at most $2D_p\ve$. Therefore,
\[\|f(x)-\rho'(x)\|_p'\leq(1+2D_p)\ve\]
for every $x$. 

Since $\ve\leq 1/8$, $(1-4\ve^p)^{-1}\leq 1+8\ve^p$. Therefore, by Lemma \ref{partialaffinetorep} there is a partial representation $\rho$ such that $\|\rho(x)-\rho'(x)\|_p\leq 8^{1/p}\ve$ for every $x$. Putting these estimates together gives the result.
\end{proof}

When $p=1$, the constant we obtain is 131, and when $p=2$ it is less than $30$. Again, the constant is decreasing in $p$, so this time a constant of 131 is valid for all $p$. 

\section{Uniqueness}

We have proved that every approximate representation can be approximated by an exact representation. In this section we prove that the representation is approximately unique in the following sense: given any two representations $\rho$ and $\sigma$ that are close, there must be a unitary map close to the identity such that $\rho$ and $U\sigma U^*$ are equal on a subspace of low codimension. 

We begin with a simple lemma that will give us a convenient way of showing that components of the representations  are equivalent.

\begin{lemma}
Let $\rho,\sigma:G\rightarrow U(n)$ be two irreducible representations such that $\|\rho(x)-\sigma(x)\|_p'<1$. Then $\rho$ and $\sigma$ are equivalent.
\end{lemma}

\begin{proof}
Let $T=\E_x\rho(x)\sigma(x)^*$. For any $y\in G$ we have:
\[\rho(y)T\sigma(y)^*= \rho(y)\left(\E_x \rho(x)\sigma(x)^*\right)\sigma(y)^*=\E_x \rho(yx)\sigma(yx)^*=\E_x \rho(x)\sigma(x)^*=T.\]
Therefore $\rho(x)T=T\sigma(x)$. By Schur's lemma it is enough to show that $T$ is not zero. This is straightforward, as 
\begin{eqnarray*}
\|I-T\|_p'&=&\|\E_x\rho(x)\rho(x)^*-\E_x \rho(x)\sigma(x)^*\|_p'\\
&=&\|\E_x \rho(x)(\rho(x)^*-\sigma(x)^*)\|_p'\\
&\leq&\E_x\|\rho(x)^*-\sigma(x)^*\|_p' < 1.\\
\end{eqnarray*}
And we are done.
\end{proof}

We now introduce a definition that we will use for our version of uniqueness of the representation approximating an approximate representation.

\begin{definition} Call a matrix $U$ $\ve$-\emph{unitary} if all its singular values are 1 or 0 and $\|UU^*-I\|_p'\leq\ve$.
\end{definition}

It is easy to check that an $n\times n$ matrix $U$ is $\ve$-unitary if and only if it can be written as $PV$ for an orthogonal projection $P$ of rank at least $(1-\ve^p)n$ and a unitary matrix $V$, which is the same as saying that all its singular values are 0 or 1 and at most $\ve^p n$ of them are 0.

\begin{theorem}\label{uniqueness} Let $\rho,\sigma:G\rightarrow U(n)$ be two representations such that $\|\rho(x)-\sigma(x)\|_p'\leq \ve$ for all $x\in G$. Then there exists a $2\ve$-unitary matrix $T'$ such that $\|T'-I\|_p'\leq3\ve$ and $\rho(x)T'=T'\sigma(x)$ for every $x$. Moreover, there is a representation $\tau$ of dimension at least $(1-(2\ve)^p)n$ that is a component of both $\rho$ and $\sigma$.
\end{theorem}

\begin{proof}
Let $T=\E_x \rho(x)\sigma(x)^*$. For each $y\in G$ we have
\[\rho(y)T\sigma(y)^*= \rho(y)\left(\E_x \rho(x)\sigma(x)^*\right)\sigma(y)^*=\E_x \rho(yx)\sigma(yx)^*=\E_x \rho(x)\sigma(x)^*=T.\]
So $\rho(x)T=T\sigma(x)$. That is, in traditional representation theory parlance, $T$ intertwines $\rho$ and $\sigma$. We also have
\begin{eqnarray*}
\|I-T\|_p'&=&\|\E_x\rho(x)\rho(x)^*-\E_x \rho(x)\sigma(x)^*\|_p'\\
&=&\|\E_x \rho(x)(\rho(x)^*-\sigma(x)^*)\|_p'\\
&\leq&\E_x\|\rho(x)^*-\sigma(x)^*\|_p' \leq \ve
\end{eqnarray*}
Also,
\[\|T\|_{\op}=\|\E_x \rho(x)\sigma(x)^*\|_{\op}\leq \E_x \|\rho(x)\sigma(x)^*\|_{\op}=1.\]
Therefore,
\begin{eqnarray*}
\|TT^*-I\|_p'&\leq& \|TT^*-T\|_p'+\|T-I\|_p'\\
& \leq& \|T\|_{\op}\|T-I\|_p'+\|T-I\|_p'\\
& \leq& 2\ve
\end{eqnarray*}
Let $\lambda_1,\lambda_2,\dots,\lambda_n\in[0,1]$ be the singular values of $T$ and let $u_1,\dots,u_n$ and $v_1,\dots,v_n$ be two orthonormal sequences such that $Tu_i=\lambda_i v_i$ for each $i$. Partition $\{1,\dots,n\}$ into sets $A_0\cup A_1\cup\dots\cup A_k$ in such a way that $\lambda_i=\lambda_j$ if and only if $i,j\in A_r$ for some $r$, taking $A_0$ to be the set $\{i:\lambda_i=0\}$. (Strictly speaking this is not a partition because we allow $A_0$ to be empty, but we do insist that the remaining $A_i$ are non-empty.) Then
\[|A_0|\le \sum_i |\lambda_i^2-1|^p=\|TT^*-I\|_p^p\leq (2\ve)^p n.\]
For each $r$, let $U_r=\mbox{span}_{i\in A_r} \langle u_i\rangle$ and let $V_r=\mbox{span}_{i\in A_r} \langle v_i\rangle$. Then $\mbox{dim}(U_r)=\mbox{dim}(V_r)=|A_r|.$

Now 
\[\rho(x)TT^*=T\sigma(x)T^*=T\sigma(x^{-1})^*T^*=TT^*\rho(x^{-1})^*=TT^*\rho(x)\]
for every $x$. Similarly $\sigma(x)T^*T=T^*T\sigma(x)$ for every $x$.

For each $i$ we have
\[TT^*v_i=\lambda_i^2v_i,\ \ \ \ \ T^*Tu_i=\lambda_i^2u_i.\]
Thus, each $U_r$ is the eigenspace of $T^*T$ corresponding to the eigenvalue $\lambda_i^2$ for any $i\in A_r$, and $V_r$ is the eigenspace of $TT^*$ corresponding to the same eigenvalue.

Since each $\rho(x)$ commutes with $TT^*$, it follows that each $V_r$ is invariant under $\rho(x)$, and similarly each $U_r$ is invariant under $\sigma(x)$.
 
For each $i$ let $u_i'=\sigma(x)u_i$ and $v_i=\rho(x)v_i$. Then
\[Tu_i'=T\sigma(x)u_i=\rho(x)Tu_i=\lambda_i \rho(x)v_i=\lambda_i v_i'.\]
Since the spaces $U_r$ and $V_r$ are invariant under $\sigma$ and $\rho$, respectively, if $u'_i\in U_r$ and $v'_i\in V_r$ for each $i\in A_r$. So $\{u'_i\}_{i\in A_r}$ and $\{v'_i\}_{i\in A_r}$ are orthonormal bases for $U_r$ and $V_r$.

For each $r$ let $T_r:U_r\rightarrow V_r$ be the linear map $T_{|U_r}$ and for $r\ne 0$ let $T_r'=\lambda_r^{-1}T_r$. Thus, $T_r'$ is the map that takes $u_i'$ to $v_i'$ for each $i\in A_r$. We also have that  
\[\rho(x)T_r'u_i=\rho(x)v_i=v_i'=T_r'u_i'=T_r'\sigma(x)u_i\]
for each $x$, so $\rho(x)T'_ru$ and $T'_r\sigma(x)u$ are equal for all $u\in U_r$. Thus, for each $r\neq 0$, the restrictions $\rho_{|V_r}$ and $\sigma_{|U_r}$ are equivalent. 

For $r=0$, define $T_0':U_0\rightarrow V_0$ to be the zero map. Then for any $u\in U_0$, we again have that $\rho(x)T_0'=T_0'\sigma(x)$. Therefore, if we define $T'$ to be $T_0'\oplus T_1'\oplus\dots\oplus T_k'$, then $\rho(x)T'=T'\sigma(x)$ for every $x\in G$. By definition, $T'$ is a map with singular values 0 and 1, and at least $(1-(2\ve)^p)n$ of those singular values are equal to 1. Therefore, $T'$ is $2\ve$-unitary.

Letting $U=\oplus_{r\neq0}U_r$ and $V=\oplus_{r\neq0}V_r$, we have that $U$ and $V$ are subspaces of dimension at least $(1-(2\ve)^p)n$, with $U$ being $\sigma$-invariant and $V$ being $\rho$-invariant. Also, $\rho(x)T'_{|U}=T'_{|U}\sigma(x)$ for every $x$, so $\rho_{|V}$ and $\sigma_{|U}$ are equivalent. This gives us the representation $\tau$ in the statement of the theorem.

Finally, note that since $(T-T')u_i=(1-\lambda_i)v_i$ for each $i$, we have the bound
\[\|T-T'\|_p'^p=\E_i |\lambda_i-1|^p\leq\E_i |\lambda_i^2-1|^p\leq (2\ve)^p.\]
Therefore,
\[\|T'-I\|_p'\leq \|T-T'\|_p'+\|T-I\|_p'\leq\ve+2\ve=3\ve.\]
This completes the proof.
\end{proof}



\section{Concluding remarks and questions}

\subsection{Reformulating our main stability results}

There are several ways of stating our main results. We have chosen to state them in terms of partial representations, since this can be done concisely, and the converse to the inverse theorem (that is, the statement that a function that correlates with a partial representation must have a large $U^2$ norm) has a natural statement and proof in these terms. However, it is worth pointing out that our results show that an approximate representation can in an appropriate sense be approximated by a representation of approximately the same dimension.

To see this, let $f:G\to U(n)$ be approximated by an $(n,m)$-partial affine representation $\rho$, and suppose first that $m\leq n$. Let $\rho(x)=VP(x)U^*$ for each $x$, where $V$ and $U$ are $n\times m$ partial unitary matrices and $P:G\to U(m)$ is a representation. Then let $Q(x)=P(x)\oplus I_{n-m}$ for each $x$ and let $U_1=(U|W)$ and $V_1=(V|Z)$ be $n\times n$ unitary matrices that extend $U$ and $V$. Then $V_1Q(x)U_1^*=VP(x)U^*+ZW^*$. Since $ZW^*$ is a matrix with $n-m$ singular values equal to 1 and the rest equal to zero, $\|ZW^*\|_p'=((n-m)/n)^{1/p}$. But the map $x\mapsto V_1Q(x)U_1^*$ is an affine representation, so we in fact have an affine representation that approximates $f$. If $U=V$ then we can ensure that $U_1=V_1$ and obtain a representation that approximates $f$.

If $m>n$, then we do not necessarily have a representation or affine representation that approximates $f$, but we do have one that approximates $f\oplus I_{m-n}$. This time, let $U_1=\begin{pmatrix} U\\W\\ \end{pmatrix}$ and $V_1=\begin{pmatrix}V\\ Z\\ \end{pmatrix}$ be $m\times m$ unitary matrices that extend $U$ and $V$. Then
\[U_1P(x)V_1=\begin{pmatrix} VP(x)U^*&VP(x)W^*\\ ZP(x)U^*&ZP(x)W^*\\ \end{pmatrix}.\]
Since each of $VP(x)W^*$, $ZP(x)U^*$ and $ZP(x)W^*$ has rank at most $m-n$ and operator norm at most 1, each has normalized $p$-norm (the normalization being in M$_m(\C)$) at most $((m-n)/m)^{1/p}$, as does $I_{m-n}$. Thus $f\oplus I_{m-n}$ is approximated by the affine representation $V_1P(x)U^*$. Again, if $U=V$ then we may obtain a representation.

\subsection{Allowing functions to take non-unitary values}

Suppose that we weaken the condition on approximate affine representations so that instead of requiring them to take unitary values we require only that they take values with operator norm at most 1. If $G$ is a finite group and $f:G\to$M$_n(\C)$ is such a map, then for every $x$ we have the inequality
\[\|I_n-f(x)f(x)^*f(x)f(x)^*\|_p'\leq\ve.\]
If the singular values of $f(x)$ are $\lambda_1,\dots,\lambda_n$, then the left-hand side is equal to $(\E_i|1-\lambda_i^4|^p)^{1/p}$, so it follows that $\E_i|1-\lambda_i|\leq\ve^p$. Therefore, we can approximate $f(x)$ to within $\ve$ by a unitary matrix $g(x)$. Then an easy triangle-inequality argument (this is where we use the fact that each $f(x)$ has operator norm at most 1) shows that 
\[\|f(x)f(y)^*f(z)f(w)^*-g(x)g(y)^*g(z)g(w^*)\|_p'\leq 4\ve\]
for any $x,y,z,w\in G$. It follows that $f$ can be approximated to within $\ve$ by an affine $4\ve$-representation that takes unitary values. If we assume only that $\|f(x)\|_\op\leq C$ for every $x$, then we obtain an affine $4C^3\ve$-representation instead. Therefore, our main stability theorem for affine $\ve$-representations holds, with a slightly worse bound, under this weaker assumption.

The situation for $\ve$-representations is not quite as straightforward. We show first that it is possible to relax the conditions when $p=1$ or $2$, but these proofs rely on specific properties of the nuclear and Hilbert-Schmidt norms. We then give an argument that works for all $p$ in the range $[1,2]$. 

\begin{lemma} \label{identity} Let $f:G\to$M$_n(\C)$ be a map such that $f(e)$ is unitary and $\|f(xy)-f(x)f(y)\|_p'\leq\ve$ for every $x,y\in G$. Then $\|f(e)-I_n\|_p'\leq\ve$.
\end{lemma}

\begin{proof}
Since $\|.\|_p'$ is unitary invariant, 
\[\|f(e)-I_n\|_p'=\|f(e)f(e)-f(e)\|_p',\]
which is at most $\ve$, by hypothesis.
\end{proof}

\begin{lemma} \label{relaxnuclear}
Let $f:G\to$M$_n(\C)$ be a map such that $f(e)$ is unitary, $\|f(x)\|_\op\leq 1$ for every $x\in G$, and $\|f(xy)-f(x)f(y)\|_\nuc'\leq\ve$ for every $x,y\in G$. Let $x\in G$ and let $g(x)$ be the unitary map obtained by replacing all the singular values of $f(x)$ by 1. Then $\|f(x)-g(x)\|_\nuc'\leq 2\ve$.
\end{lemma}

\begin{proof}
We know that $\|f(x)f(x^{-1})-f(e)\|_\nuc'\leq\ve$, and therefore, by Lemma \ref{identity}, that $\|f(x)f(x^{-1})-I_n\|_\nuc'\leq 2\ve$. It follows that $\|f(x)f(x^{-1})\|_\nuc'\geq 1-2\ve$, and therefore, since $\|f(x)^{-1}\|_\op\leq 1$, that $\|f(x)\|_\nuc'\geq 1-2\ve$. It follows that $\|f(x)-g(x)\|_\nuc'\leq 2\ve$.
\end{proof}

We now prove the same thing for the normalized Hilbert-Schmidt norm.

\begin{lemma} \label{relaxhs}
Let $f:G\to$M$_n(\C)$ be a map such that $f(e)$ is unitary, $\|f(x)\|_\op\leq 1$ for every $x\in G$, and $\|f(xy)-f(x)f(y)\|_{hs}\leq\ve$ for every $x,y\in G$. Let $x\in G$ and let $g(x)$ be the unitary map obtained by replacing all the singular values of $f(x)$ by 1. Then $\|f(x)-g(x)\|_{hs}\leq 2\ve$.
\end{lemma}

\begin{proof}
Let the singular-value decomposition of $f(x)$ be $\sum_i\lambda_ia_i\otimes\overline{b_i}$. Then for any matrix $T$ we have $f(x)T=\sum_i\lambda_ia_i\otimes\overline{T^*b_i}$, and
\[f(x)T-I_n=\sum_i(\lambda_ia_i\otimes\overline{T^*b_i}-a_i\otimes b_i).\]
Because the $a_i$ are orthogonal, so are the rank-1 matrices $\lambda_i a_i\otimes\overline{T^*b_i}-a_i\otimes b_i$.
If we suppose in addition that $\|T\|_\op\leq 1$, then $\|\lambda_ia_i\otimes\overline{T^*b_i}\|_{hs}\leq\lambda_i$, from which it follows that
\[\|\lambda_i a_i\otimes\overline{T^*b_i}-a_i\otimes b_i\|_{hs}\geq 1-\lambda_i.\]
It follows that $\|f(x)T-I_n\|_{hs}^2\geq\E_i(1-\lambda_i)^2$.

Now let us apply this to the matrix $T=f(x^{-1})$. As in the proof of Lemma \ref{relaxnuclear} we have that $\|f(x)T-I_n\|_{hs}\leq 2\ve$, from which it follows that $\E_i(1-\lambda_i)^2\leq 4\ve^2$ and hence that $\|f(x)-g(x)\|_{hs}\leq 2\ve$, as claimed.
\end{proof}

\begin{lemma}
Let $f:G\to$M$_n(\C)$ be a function such that $\|f(x)\|_\op\leq 1$ for every $x\in G$, and $\|f(xy)-f(x)f(y)\|_p'\leq\ve$ for every $x,y\in G$. Let $g:G\to U(n)$ be a function with $\|g(x)-f(x)\|_p'\leq\d$ for every $x\in G$. Then $\|g(xy)-g(x)g(y)\|_p'\leq\ve+3\d$ for every $x,y\in G$.
\end{lemma}

\begin{proof}
This follows straightforwardly from the triangle inequality and the fact that all the maps have operator norm at most 1.
\end{proof}

Putting these lemmas together, we see that if $p=1$ or $2$, then Theorem \ref{reppcase} holds (with a larger constant) even if we just assume that $f(e)$ is unitary and $\|f(x)\|_\op\leq 1$ for every $x$. In order to obtain the same result for all $p$ in between, we need to prove the plausible result that for every matrix $A$ with $\|A\|_\op=1$, the matrix $T$ that minimizes the distance $\|AT-I_n\|_p'$ amongst all matrices with operator norm at most 1 is the matrix $U^*$, where $U$ is the matrix obtained from $A$ by replacing all its singular values with 1. 

We did not ourselves see how to do this -- we are grateful to Suvrit Sra for supplying a proof on Mathoverflow (\url{http://mathoverflow.net/questions/204580/on-closest-unitary-matrix}). He has kindly allowed us to include it here.

For a given matrix $X$, let $s_j(X)$ denote the $j$-th singular value of a matrix $X$ in decreasing order. Similarly, let $\lambda_j(X)$ denote the $j$-th eigenvalue of a Hermitian matrix $X$. Let $S(X)$ denote the diagonal matrix of singular values of $X$.

\begin{lemma}\label{eigenvalues}
Let $A,B$ be Hermitian matrices such that $A\geq B$. Then $$\lambda_k(A)\geq \lambda_k(B)$$
for all $k$.
\end{lemma}
\begin{proof}
There is a positive-semidefinite matrix $X$ such that $A=B+X$. Hence for any vector $u$
$$\langle Au,u\rangle= \langle Bu,u\rangle + \langle Xu,u\rangle\geq \langle Bu,u\rangle$$
Now let $u_1,\dots, u_n$ be an orthonormal sequence such that $u_i$ is an eigenvector of $A$ with eigenvalue $\lambda_i(A)$. Similarly, let $v_1,\dots, v_n$ be an orthonormal sequence such that $v_i$ is an eigenvector of $B$ with eigenvalue $\lambda_i(B)$. Let $V_k=\mbox{span}\langle v_1,\dots, v_k\rangle $ and $U_k=\mbox{span}\langle u_k,\dots, u_n \rangle$. Since $\dim U_k + \dim V_k = n+1$, there is a non-zero vector $w\in U_k\cap V_k$. Let $w=\alpha_1 v_1+\dots \alpha_k v_k = \beta_k u_k + \dots \beta_n u_n$.
Note that
$$\langle Bw, w \rangle= \sum_{i=1}^k \alpha_i^2\lambda_i(B) \geq \lambda_k (B) \sum_{i=1}^k \alpha_i^2=\lambda_k(B) \|w\|^2$$
and 
$$\langle Aw, w \rangle= \sum_{i=1}^k \beta_i^2\lambda_i(A) \leq \lambda_k (A) \sum_{i=1}^k \beta_i^2=\lambda_k(A) \|w\|^2.$$
Comparing the two inequalities we get
$$\lambda_k(A)\|w\|^2 \geq \langle Aw, w \rangle \geq \langle Bw , w\rangle \geq \lambda_k(B) \|w\|^2$$
and we are done, since $w\ne 0$.
\end{proof}

\begin{lemma}\label{unitaryapproximation}
 Let $A,B$ be $n\times n$ matrices of operator norm at most 1 and let $U$ be the unitary matrix obtained by replacing all singular values of $A$ by 1. Then
$$\|A-U\|_p' \leq \|AB-I\|_p'.$$
\end{lemma}
\begin{proof}
Since $B$ has operator norm at most 1, $I\geq BB^*$. Hence, $AA^*\geq ABB^*A^*$. Therefore, by Lemma \ref{eigenvalues}, $\lambda_k(AA^*)\geq \lambda_k(ABB^*A^*)$ for all $k$. It follows that
\begin{eqnarray*}
  s_k(A) = \lambda_k^{1/2}(AA^*) \ge \lambda_k^{1/2}(ABB^*A^*) = s_k(AB).
\end{eqnarray*}
Since $A$ has operator norm at most 1, we have $0 \le 1 - s_k(A) \le 1 - s_k(AB)$ for all $k$. Consequently, it follows that
\begin{equation*}
  \|I-S(A)\|_p' \le \|I-S(AB)\|_p'.
\end{equation*}
Now using a corollary of Lidskii's majorization theorem (see e.g., Theorem IV.3.4 in \cite{Bhatia}), it follows that
\begin{equation*}
  \|I-S(AB)\|_p' = \|S(I)-S(AB)\|_p' \le \|I-AB\|_p'.
\end{equation*}
Finally we have
\begin{equation*}
  \|A-U\|_p' = \|I-S(A)\|_p' \le \|I-S(AB)\|_p'\leq \|I-AB\|_p',
\end{equation*}
and we are done.
\end{proof}

\begin{lemma}\label{relaxschatten}
Let $f:G\to$M$_n(\C)$ be a map such that $f(e)$ is unitary, $\|f(x)\|_\op\leq 1$ for every $x\in G$, and $\|f(xy)-f(x)f(y)\|_p'\leq\ve$ for every $x,y\in G$. Let $x\in G$ and let $g(x)$ be the unitary map obtained by replacing all the singular values of $f(x)$ by 1. Then $\|f(x)-g(x)\|_p'\leq 2\ve$.
\end{lemma}
\begin{proof}
We know that $\|f(x)f(x^{-1})-f(e)\|_p'\leq\ve$, and therefore, by Lemma \ref{identity}, that $\|f(x)f(x^{-1})-I_n\|_p'\leq 2\ve$. Hence by Lemma \ref{unitaryapproximation}, if $g(x)$ is the unitary map obtained by replacing all the singular values of $f(x)$ by 1, then $\|f(x)-g(x)\|_p' \leq 2\ve$.
\end{proof}

\subsection{Why does our proof work less well when $2<p<\infty$?}

We have proved stability theorems for the Schatten $p$-norms when $1\leq p\leq 2$, and the theorem of Grove, Karcher and Ruh gives us the corresponding results when $p=\infty$. In all cases, the bound we obtain for the maximum distance between $f(x)$ and the approximating partial representation (or affine representation) depends linearly on the initial parameter $\ve$.

For $2<p<\infty$ we can deduce stability theorems, but we lose the linear dependence. For example, if $f$ is an affine $\ve$-representation with respect to the norm $\|.\|_p'$, then it is also an affine $\ve$-representation with respect to the norm $\|.\|_{hs}$ (since the normalized Schatten $p$-norms increase with $p$). It follows from Theorem \ref{affinepcase} that there is a partial affine representation $\rho$ such that $\|f(x)-\rho(x)\|_{hs}\leq 13\ve$ for every $x$. Since $\|f(x)-\rho(x)\|_\op\leq 2$, it follows from this that $\|f(x)-\rho(x)\|_q'\leq 2^{1-2/p}(13\ve)^{2/p}$. Thus, we obtain a bound of $C\ve^{2/p}$ for an absolute constant $C$. 

As $p\to\infty$, this bound becomes less and less informative, and when $p=\infty$ it tells us nothing at all. And yet we know from the result of Grove, Karcher and Ruh that the result is true with a linear bound when $p=\infty$. 

It is a little mysterious that this should be the case, since, as we mentioned in the introduction, Narutaka Ozawa has informed us of an alternative argument that works uniformly across the entire range $[1,\infty]$. We would like to know whether some modification of our argument could be used for higher $p$. The main point at which our proof currently breaks down when $p>2$ is Lemma \ref{distancefromid}. When $p>2$, the best we can say about $\Re\,\tr'(A)$ when $\|A-I_n\|_p'\leq\ve$ is that it is at least $1-C\ve^2$, rather than $1-C\ve^p$, which is what we would need to obtain a linear bound.

\subsection{What is the correct power in the inverse theorem?}

We showed that if $f:G\to$M$_n(\C)$ is a map with $\|f(x)\|_\op\leq 1$ for every $x$ and $\|f\|_{U^2}^4\geq cn$, then there is a partial affine representation $\rho$ with dimension $m\in [cn/2,2n/c]$ such that $\E_x\langle f(x),\rho(x)\rangle \geq c^2m/16$. If $G$ is an Abelian group and $n=1$, then we have the inequality 
\[\|f\|_{U^2}^4=\sum_r|\hat{f}(r)|^4\leq\max_r|\hat f(r)|^2,\]
which implies that there is a character $\chi$ with $|\E_xf(x)\overline{\chi(x)}|\geq c^{1/2}$. Thus, our argument does not give the correct bound in this case. A possible explanation for the discrepancy is that if $n=1$, then Corollary \ref{bothpartialunitary} holds trivially, since the conditions force $V$ and $U$ to be partial unitary matrices already, without the need to pass from $c$ to $c^4$. 

We do not know whether there is a genuine difference here (which might be the case, given that the representation that correlates with $f$ sometimes has to have dimension considerably larger than that of $f$), or whether there are inefficiencies in our argument. Probably both are true. In any case, it would be interesting to work out the right exponent in the dependence on $c$.

\subsection{Generalizing to compact groups}

We have proved stability theorems when $G$ is a finite group. A natural question is whether the same result is true for other groups. This is the case when we have a suitable Fourier analysis on $G$. In particular, it is true if $G$ is compact, when our results generalize straightforwardly.

Indeed, let $G$ be a compact group with Haar measure $\mu$. Let us write $\hat G$ for the set of all irreducible representations of $G$, which is a discrete set. Then all the definitions and proofs are more or less unchanged, except that averages over $G$ become integrals with respect to Haar measure. For example, the Fourier transform of the matrix-valued function $f:G\rightarrow$M$_n(\C)$ is given by the formula
$$\hat f(\rho)=\int_{x\in G}f(x)\otimes \overline{\rho(x)} d\mu(x).$$
Parseval's identity is
$$\int_x \mbox{tr}(f(x)g(x)^*)d\mu(x)=\sum_{\rho\in \hat G} n_\rho\tr(\hat f(\rho)\hat g(\rho)^*)d\hat\mu(\rho).$$
The Fourier inversion formula is
$$f(x)=\sum_{\rho\in \hat G} n_\rho \tr_\rho \bigl(\overline{\rho(x^{-1})}\cdot \hat f(\rho) \bigr).$$
We also have the same Fourier interpretation for the $U^2$ norm.
$$\|f\|_{U^2}^4=\int_{xy_1^{-1}zt^{-1}=e} \tr(f(x)f(y)^*f(z)f(w)^*) = \sum_{\rho\in \hat G} n_\rho \|\hat f(\rho)\|_\square^4$$
With these small modifications, one can obtain our main results with the same bounds for measurable matrix-valued functions on compact groups.

\end{document}